\documentclass[11pt]{amsart}
\usepackage{amsmath,amssymb,amsrefs}
\usepackage{enumitem}
\usepackage{hyperref}
\usepackage{color}

\newtheorem{maintheorem}{Theorem}

\newtheorem{theorem}{Theorem}[section]
\newtheorem{proposition}[theorem]{Proposition}
\newtheorem{corollary}[theorem]{Corollary}
\newtheorem{lemma}[theorem]{Lemma}

\theoremstyle{definition}
\newtheorem{definition}[theorem]{Definition}
\newtheorem{remark}[theorem]{Remark}

\newcommand{\R}{\mathbb{R}}
\newcommand{\owedge}{\mathbin{\bigcirc\mspace{-15mu}\wedge\mspace{3mu}}}

\numberwithin{equation}{section}

\addtocontents{toc}{\setcounter{tocdepth}{2}}

\begin{document}

\title{Bandwidth and focal radius with positive isotropic curvature}

\author{Tsz-Kiu Aaron Chow}
\address{Department of Mathematics, Massachusetts Institute of Technology, Cambridge MA 02139, USA}
\thanks{T.-K. A. C. was supported by the Croucher Foundation.} 
\email{\href{chowtka@mit.edu}{chowtka@mit.edu}}

\author{Jingze Zhu}
\address{Department of Mathematics, Massachusetts Institute of Technology, Cambridge MA 02139, USA}
\email{\href{zhujz@mit.edu}{zhujz@mit.edu}}

\maketitle

\begin{abstract}
This paper investigates quantitative metric inequalities for manifolds with positive isotropic curvature (PIC). Our results include upper bounds on the bandwidth and focal radius of hypersurfaces in PIC manifolds, contingent on boundary convexities and Betti numbers. The proof is based on exploiting the spectral properties of a twisted de Rham-Hodge operator on manifolds with boundary. 
\end{abstract}

\tableofcontents

\section{Introduction}
Our goal in this paper is to study quantitative metric inequalities for manifolds with positive isotropic curvature (PIC). PIC is a natural curvature condition with fruitful topological implications. Micallef and Moore \cite{MM} first introduced this condition in their study of stable minimal two-spheres, proving that a closed, simply-connected Riemannian manifold with PIC and dimension at least 4 is homeomorphic to a sphere. Subsequently, Fraser \cite{Fra03} extended this idea to stable minimal surfaces of genus 1, demonstrating that the fundamental group of a closed PIC manifold cannot contain a free abelian subgroup of rank greater than 1. On the other hand, by exploiting the Weitzenb\"ock formula for two-forms, Micallef-Wang \cite{MW93} and Seaman \cite{Sea} independently showed that the second Betti number is an obstruction to PIC on closed manifolds of even dimensions. In fact, they showed that the curvature term in the Weitzenb\"ock formula for two-forms can be expressed in terms of isotropic curvatures. PIC also plays a central role in the proof of the Differentiable Sphere Theorem \cite{BS09, BS08}. Using Ricci flow with surgery, Brendle \cite{Bre19} obtained a complete classification of closed PIC manifolds of dimension at least 12 under a mild topological assumption. This was previously accomplished in dimension 4 by Hamilton \cite{Ham97} and Chen-Zhu \cite{CZ06}. Let us recall the relevant definitions:

\begin{definition}
	We denote by $C_{PIC}$ the set of all algebraic curvature tensors $R$ which have nonnegative isotropic curvature in the sense that
		\[	R_{1313} + R_{1414} + R_{2323} + R_{2424} - 2R_{1234} \geq 0\]
	for all orthonormal four-frames $\{e_1, e_2, e_3, e_4\}$.	

\end{definition}

\begin{definition}\label{def;PIC}
	An algebraic curvature tensor $R$ is said to have isotropic curvature uniformly bounded below by $\sigma \geq 0$ if
		\[	R_{1313} + R_{1414} + R_{2323} + R_{2424} - 2R_{1234} \geq \sigma\]
		for all orthonormal four-frames $\{e_1, e_2, e_3, e_4\}$. 	In other words,
	\[	R - \frac{1}{8}\sigma g\owedge g \in C_{PIC}.\]
\end{definition}

To our knowledge, there are three main approaches in studying the topology of manifolds admitting PIC: the application of stability of minimal surfaces, Ricci flow, and the Weitzenb\"ock formula for two-forms. This is similar to the study of positive scalar curvature (PSC). However, a major difference is that, in the study of PSC, the Weitzenb\"ock formula for spinors, rather than two-forms, is used to exploit the spectral properties of the Dirac operator on manifolds with PSC. By exploiting the Weitzenb\"ock formula for two-forms on manifolds with boundary, our first result extends the one in \cite{MW93} to manifolds with boundary:

\begin{maintheorem}\label{main;thm;0}
	Let $(M^n, g), n\geq 4$ be an even-dimensional compact Riemannian band with boundary $\partial M$. 	 Suppose that $(M, g)$ has positive isotropic curvature, then:
	\begin{itemize}
		\item [(i)] If $\partial M$ is weakly 2-convex in the sense that
		\[A(X,X) + A(Y, Y) \geq 0 \]
			for all orthonormal $X, Y\in T(\partial M)$, then $b_2(M) = 0$.
		\item[(ii)] If $\partial M$ is weakly $(n-2)$-convex in the sense that 
				\[A(X_1,X_1) +\cdots + A(X_{n-2}, X_{n-2}) \geq 0 \]
			for all orthonormal $X_1,\hdots, X_{n-2}\in T(\partial M)$, then $b_{n-2}(M) = 0$.
	\end{itemize} 
The boundary second fundamental form $A$ is defined by $A(X,Y) = g(X,\ \nabla_{\nu}Y)$ for all  $X, Y\in T(\partial M)$, where $\nu$ is the unit outward normal vector field.
\end{maintheorem}

We note that Theorem \ref{main;thm;0}(i) is implied by \cite{Fra02}, which proved that a two-convex domain in a manifold with positive isotropic curvature is contractible.

Recently, Gromov \cite{Gro19, Gro4lecture} proposed to study the geometry of PSC through quantitative metric inequalities, akin to classical Riemannian comparison geometry. Using the method of minimal slicing, Gromov \cite{Gro19} proved that if $g$ is a Riemannian metric on $T^{n-1}\times [-1, 1]$ with $n\leq 7$ and scalar curvature bounded below by a positive constant $R_g\geq \sigma > 0$, then the width of the manifold has an upper bound depending on the positive lower bound of the scalar curvature: $\text{dist}_g(T^{n-1}\times \{-1\},\, T^{n-1}\times \{1\}) \leq  2\pi \sqrt{\frac{n-1}{n\sigma}}$. Following Gromov's proposal, Cecchini \cite{Cecc}, Cecchini-Zeidler \cite{CZ19} and Zeidler \cite{Zeid} extended the method of the Dirac operator to study the quantitative bandwidth of spin Riemannian bands. A Riemannian band is a Riemannian manifold $(V, g)$  with two connected boundary components $\partial V = \partial_-V\sqcup \partial_+V$, generalizing the notion of a cylinder. It was proved in \cite{CZ19} that if $(V, g)$ is an odd-dimensional $\hat{A}$-overtorical spin band with scalar curvature bounded below $R_g\geq \sigma > 0$, then the bandwidth satisfies $\text{dist}_g(\partial_-V,\, \partial_+V) \leq  2\pi \sqrt{\frac{n-1}{n\sigma}}$. The key idea in \cite{CZ19} is  to study the spectral properties of a Callias operator, which is the Dirac operator perturbed by a potential function that depends only on the distance from the boundary components. 

Inspired by the ideas in \cite{CZ19}, we aim to adapt the Callias operator method to the setting of manifolds with positive isotropic curvature under mild topological assumptions. A key observation is that the de Rham-Hodge operator $d + d^*$ is a Dirac-type operator on the exterior bundle \cite[Theorem 5.12]{LM}. Introducing a potential function into the de Rham-Hodge operator is delicate. 
Since the PIC condition appears only in the curvature term of the Weitzenb\"ock formula for two-forms \cite{MW93}, we cannot simply perturb the de Rham-Hodge operator by adding a potential function. Instead, we seek to twist the de Rham-Hodge operator so that the square of the twisted operator maps $k$-forms to $k$-forms, while preserving the applicability of Hodge theory. We adapt Witten's approach \cite{Wit82}: given a smooth function $f: M\to \R$, we twist the exterior differential $d$ and its adjoint $d^*$ by $d_f = e^{-f}d e^f$ and $d^*_f = e^{f}d^* e^{-f}$. Witten showed \cite{Wit82} that on a closed manifold, the dimension of the kernel of  $d_fd^*_f + d^*_fd_f$ is the same as that of $dd^* + d^*d$. In our case, for manifolds with boundary, we will show that this remains true using Hodge theory with boundary. This motivates us to consider the Callias-type operator
	\[	D_f = d_f + d^*_f\]
to study quantitative metric inequalities on PIC manifolds with a non-vanishing Betti number. In Section \ref{section;integral;estimate}, we will show that the Weitzenb\"ock formula for this operator is
	\[	D_f^2 = -\Delta + \mathcal{R}  + 2 \sum_{i,j=1}^n (\nabla_{e_i,e_j}^2 f) \theta^i\wedge i_{e_j}  -\Delta f + |\nabla f|^2.\]
Moreover, following the approach of \cite{MW93}, we will show that if the manifold $(M, g)$ is even-dimensional and has isotropic curvature bounded below by a positive constant $\sigma$ in the sense of Definition \ref{def;PIC}, then the curvature endomorphism $\mathcal{R}$ satisfies $\mathcal{R}\geq \frac{n-2}{2}\sigma$ on the space of two-forms. To that end, our main result in the case of bandwidth is

\begin{maintheorem}\label{main;thm;1}
	Let $M^n, n\geq 4$ be an even-dimensional compact Riemannian band with boundary $\partial M = \partial_-M\sqcup\partial_+M$. Let $\sigma > 0$ be a positive number and $\delta$ be a non-negative number satisfying
	  \[0\leq \delta  <  \min\left\{\frac{(n-2)\sigma}{10(n-1)\Lambda},\, \frac{(n-3)\sigma r_f}{8(n+1)},\, \frac{\sqrt{\sigma}}{2}\right\},\]
	 where $\Lambda :=  \sqrt{\max\left\{-\frac{1}{n-1}Ric_M,\, 0\right\}}$ and $r_f$ denotes the focal radius of the boundary $\partial M$ in $M$ (see Definition \ref{def;focal;radius}). 
	 Suppose that the following conditions hold:
	\begin{itemize}
		\item [(i)] The curvature tensor satisfies $R - \frac{1}{8}\sigma g\owedge g \in C_{PIC}$;
		\item[(ii)] Either $b_2(M)\neq 0$ and 
			\[A(X,X) + A(Y, Y) \geq - \delta \]
			for all orthonormal $X, Y\in T(\partial M)$;
			or $b_{n-2}(M)\neq 0$ and 
			\[A(X_1,X_1) +\cdots + A(X_{n-2}, X_{n-2}) \geq - \delta \]
			for all orthonormal $X_1,\hdots, X_{n-2}\in T(\partial M)$.  Here, the boundary second fundamental form $A$ is defined by $A(X,Y) = g(X,\ \nabla_{\nu}Y)$ for all  $X, Y\in T(\partial M)$, where $\nu$ is the unit outward normal vector field. 
	\end{itemize} 
	Then the width of $M$ satisfies
	\[	\text{dist}_g(\partial_-M,\, \partial_+M)\leq  \,\frac{51}{\sqrt{\sigma}}\, \tan^{-1}\left(\frac{\delta}{\sqrt{\sigma}}\right).\]
\end{maintheorem}

It is worth noting that, unlike the bandwidth results in the positive scalar curvature case \cite{CZ19, Gro19}, our result for positive isotropic curvature requires a boundary curvature condition. To justify the necessity for a boundary condition, we provide a counterexample in Section \ref{section;counter}, 
illustrating that under our mild topological assumptions, the bandwidth could be arbitrarily large without imposing a boundary condition.  On the other hand, it is interesting to compare the effects of the relationship between curvature  and topology on bandwidth. In the case of positive Ricci curvature, no topological assumptions are required for bandwidth \cite{HKKZ}. For positive scalar curvature, stronger topological assumptions are necessary, such as overtorical \cite{Gro19} or $\hat{A}$-overtorical \cite{CZ19}; Finally, for positive isotropic curvature, a mild topological assumption on Betti numbers is required. We also note that recently, Fraser and Schoen \cite{FS23} obtained a quantitative upper bound on the systole of stable minimal tori in terms of a positive lower bound on the isotropic curvature, thereby giving a new proof to the result in \cite{Fra03}.

Our next goal is to study the focal radius of a hypersurface in manifolds with positive isotropic curvature and nonnegative sectional curvature. Let us recall the definition of focal radius of a hypersurface:

\begin{definition}\label{def;focal;radius}
	Let $(M,g)$ be a Riemannian manifold and $\Sigma$ a smooth oriented hypersurface in $M$ with a (outward) unit normal vector field $\nu$. Then the focal radius of $\Sigma$ with respect to the normal $\nu$ is
	\[	r_f(\Sigma, M) := \sup\left\{r\Big|\ \begin{aligned}
	    &\exp_{(\cdot)}(-r\nu(\cdot)): \Sigma\to M\ \text{is defined},\\ &\text{injective and has no critical point}.
	\end{aligned} \right\}\] 
 In other words, it is the largest distance from $\Sigma$ in the direction of $-\nu$ such that all points within the radius admit a unique minimizing geodesic to $\Sigma$.
\end{definition}

 For manifolds with positive sectional curvature, Gromov proved \cite{Gro19} that if $\Sigma^{n-1}\subset S^n$ is an immersed torus in the standard $S^n$, then the focal radius of $\Sigma$ in $S^n$ satisfies $r_f(\Sigma)\leq \pi/n$. In the case of dimension 3, Zhu \cite{Zhu21} proved the sharp inequality and also obtained rigidity. More recently, It was proved in \cite{HKKZ} that if $(M^3 ,g)$ is a closed Riemannian manifold with $2$-Ricci curvature at least 4 and $\Sigma^2\subset M^3$ is a connected embedded closed surface of positive genus, then $r_f(\Sigma)\leq \pi/4$.
 
 Now we turn to even-dimensional manifolds with positive isotropic curvature and nonnegative sectional curvature. If $M$ is such a manifold with boundary $\partial M$ and $\nu$ is the outward unit normal on $\partial M$, we define the focal radius of $\partial M$ in $M$ to be $r_f(\partial M) := r_f(\partial M, \nu)$. In this case, we obtain an upper bound for the focal radius $r_f(\partial M)$ in terms of the positive lower bound of isotropic curvature:

\begin{maintheorem}\label{thm;PIC;focalradius;b2}
		Let $M^n, n\geq 4$ be an even-dimensional compact Riemannian manifold with boundary $\partial M$.  Let $\sigma > 0$ be a positive number. Assume that the following conditions hold:
		\begin{itemize}
			\item [(i)] The curvature tensor satisfies $R - \frac{1}{8}\sigma g\owedge g \in C_{PIC}$;
			\item [(ii)] The sectional curvature is nonnegative $Sec_M \geq 0$.
		\end{itemize} 
Furthermore,	suppose that either $b_2(M)\neq 0$ or $b_{n-2}(M)\neq 0$. 
	Then the focal radius $r_f$ of $\partial M$ satisfies
	\[	r_f(\partial M,\, M) \leq 9\sqrt{\frac{n}{\sigma}}.\]
\
\end{maintheorem}

Lastly, if $\Sigma$ is a smooth, oriented, embedded, closed hypersurface in $M$ with a non-vanishing 2nd or $(n-2)$th Betti number, then we also obtain an upper bound for the focal radius of $\Sigma$ in $M$:

\begin{maintheorem}\label{thm;PIC;focalradius;hypersurface;b2}
		Let $M^n, n\geq 4$ be an even-dimensional compact Riemannian manifold. Let $\sigma > 0$ be a positive number. Assume that the following conditions hold:
		\begin{itemize}
			\item [(i)] The curvature tensor satisfies $R - \frac{1}{8}\sigma g\owedge g \in C_{PIC}$;
			\item [(ii)] The sectional curvature is nonnegative $Sec_M \geq 0$.
		\end{itemize}
		Let $\Sigma\subset M$ be an oriented hypersurface. Suppose that either $b_2(\Sigma)\neq 0$ or $b_{n-2}(\Sigma)\neq 0$.
	Then the focal radius $r_f$ of $\Sigma$ satisfies
	\[	r_f(\Sigma,\, M) \leq 9\sqrt{\frac{n}{\sigma}}.\]
\
\end{maintheorem}

This paper is organized as follows: In Section \ref{section;boundary-value;problem} we study the boundary value problem for the twisted de Rham-Hodge operator with an elliptic boundary condition. After obtaining the existence of a solution, we apply the Weitzenb\"ock formula on two-forms to derive integral estimates in Section \ref{section;integral;estimate}. The resulting integral estimate contains a term involving the lower bound of isotropic curvature, as well as a boundary term involving boundary convexities. As a byproduct of the integral estimate, we proved Theorem \ref{main;thm;0} in the final part of Section \ref{section;integral;estimate}. Combining the integral estimate with a Hessian estimate derived in Appendix \ref{section;comparison}, we will prove Theorem \ref{main;thm;1} in Section \ref{section;bandwidth}, as well as Theorem \ref{thm;PIC;focalradius;b2} and Theorem \ref{thm;PIC;focalradius;hypersurface;b2} in Section \ref{section;focal;radius}. Lastly, Section \ref{section;counter} contains a counterexample to justify the necessity of imposing a boundary condition in the bandwidth estimate of Theorem \ref{main;thm;1}.

\subsubsection*{Acknowledgement} The authors would like to thank Professor Simon Brendle and Tristan Ozuch for inspiring discussions.

\bigskip

\section{A boundary-value problem for twisted de Rham-Hodge operator}\label{section;boundary-value;problem}
Let $(M, g)$ be a compact Riemannian manifold with boundary $\partial M$. We denote by $\Omega^k(M) = \wedge^k T^*M\otimes\mathbb{C}$ and $\Omega(M) = \wedge T^*M\otimes\mathbb{C}$ the complexified exterior bundle. Let $\langle\cdot,\cdot\rangle$ be the natural Hermitian metric on $\Omega(M)$ induced by $g$.	Let $\nu$ be the outward-pointing unit normal to $\partial M$. Then $\Omega^k(M)|_{\partial M}$ has the following orthogonal decomposition:
	\[	\Omega^k(M)|_{\partial M} = \left(	 \wedge^k T^*\partial M\otimes\mathbb{C}\right) \oplus \left(	\mathbb{C}\nu^{\flat} \oplus \wedge^{k-1}T^*\partial M\otimes\mathbb{C}\right).\]
	Hence, we may decompose $\omega\in\Omega^k(M)|_{\partial M}$ uniquely as $\omega = t(\omega) + n(\omega)$ with $t(\omega)\in \wedge^k T^*\partial M \otimes\mathbb{C}$ and $n(\omega)\in \mathbb{C}\nu^{\flat} \oplus \wedge^{k-1}T^*\partial M\otimes\mathbb{C}$. Note that the condition $t(\omega) = 0$ is called the relative boundary condition, and $n(\omega) = 0$ is called the absolute boundary condition. Following the treatment in \cite[Chapter 2]{Schwarz}, we consider the following two subspaces of $ \Omega^k(M)$:
	\[	\Omega_D^k(M) = \{ \omega\in \Omega^k(M)|\, t(\omega) = 0\ \text{on}\ \partial M\}\]
	and
	\[	\Omega_N^k(M) = \{ \omega\in \Omega^k(M)|\, n(\omega) = 0\ \text{on}\ \partial M\}.\]	
	
	Let $f:M\to\R$ be a smooth function. In this paper, we consider a twisted de Rham-Hodge Laplace operator acting on two-forms with the relative boundary condition or the absolute boundary condition. To formulate the boundary value problem, we follow Witten's approach \cite{Wit82} and define a twisted exterior differential by	
	\[	d_f := e^{-f} d e^f : \Omega^k(M) \to \Omega^{k+1}(M),\]
	 and the corresponding co-differential by	 
	\[	d^*_f = e^f d^* e^{-f} :\Omega^k(M) \to \Omega^{k-1}(M).\]
	Note that $d_f^2 = (d_f^*)^2 = 0$, and that $d_f^*$ is the $L^2$-adjoint of $d_f$ on the subspaces $\Omega_D^k(M)$ and $\Omega_N^k(M)$. Then	
	\[	\Delta_{H,k}^f = (d_fd^*_f + d^*_f d_f)|_{\Omega^k(M)} : \Omega^k(M) \to \Omega^k(M) \]
	are twisted Hodge-Laplacians on the space of $k$-forms. Now, we define
	\[	\mathcal{H}^k_{D,f} (M) = \left\{\omega\in \Omega_D^k(M) |\ \Delta_{H,k}^f\omega = 0 \right\}\]
	and
	\[	\mathcal{H}^k_{N,f} (M) = \left\{\omega\in \Omega_N^k(M) |\ \Delta_{H,k}^f\omega = 0 \right\}.\]

Our idea relies on the fact that the complexified exterior bundle $\Omega(M)$ is canonically a Clifford bundle and the operator $d + d^*$ is a Dirac type operator. To that end, we can define two types of Clifford multiplications on $\Omega(M)$.

\begin{definition}
	Define $c, \tilde{c}: TM\to \text{End}(\Omega(M))$ by
	\[c(\xi)\omega := \xi^{\flat}\wedge\omega - i_{\xi}\omega,\ \text{and}\quad \tilde{c}(\xi)\omega =  \xi^{\flat}\wedge\omega + i_{\xi}\omega\]
	for all $\xi\in TM$. 
\end{definition}
	If $\{e_1,\hdots, e_n\}$ is an orthonormal frame of $M$, then we have the relations \cite[P. 148]{BGV}
	\begin{align}\label{Clifford;relations}
		\begin{cases}
			&c(e_i)\tilde{c}(e_j) + \tilde{c}(e_j)c(e_i) = 0\\
			&\tilde{c}(e_i)\tilde{c}(e_j) + \tilde{c}(e_j)\tilde{c}(e_i) = 2\delta_{ij}\\
			&c(e_i)c(e_j) + c(e_j)c(e_i) = -2\delta_{ij}
		\end{cases}.
	\end{align}
	Moreover, $c(e_i)$ satisfies $\langle c(e_i)\alpha, \beta\rangle = -\langle \alpha,c(e_i)\beta\rangle$ and $\tilde{c}(e_i)$ satisfies $\langle \tilde{c}(e_i)\alpha, \beta\rangle = \langle \alpha, \tilde{c}(e_i)\beta\rangle$.

\begin{definition}
	We denote 
	\[	D = d + d^* = \sum_i c(e_i)\nabla_{e_i},\]
	which is a Dirac-type operator, hence $\Delta_H = D^2$ is the usual Hodge Laplacian. And we denote
	\[	D_f = d_f + d^*_f = D +  \tilde{c}(\nabla f),\]
	so that the twisted Hodge Laplacian can be written as $\Delta_{H}^f = D_f^2$.
\end{definition}

\begin{definition}
	On $\partial M$, we define a map $\chi: \Omega^k(M)|_{\partial M}\to \Omega^k(M)|_{\partial M}$ by
		\[	\chi \omega = \tilde{c}(\nu) c(\nu)\omega.\] 
	Note that the boundary map $\chi$ is self-adjoint with respect to the inner product $\langle \cdot,\cdot  \rangle$ and satisfies $\chi^2 = \text{id}$.
\end{definition}

\begin{lemma}\label{boundary;correspondence}
	We have
	\[	\Omega_D^{k}(M) = \{ \omega\in \Omega^{k}(M)|\, \chi\omega = -\omega\ \text{on}\ \partial M\}\]
	and
	\[	\Omega_N^{k}(M) = \{ \omega\in \Omega^{k}(M)|\,  \chi\omega = \omega\ \text{on}\ \partial M\}.\]	
\end{lemma}

\begin{proof}
	Let $\{e_1,\hdots, e_{n-1},\nu\}$ be an orthonormal frame on $\partial M$ and $\{\theta^1,\hdots, \theta^{n-1}, \nu^{\flat}\}$ the dual frame. If $\omega\in \Omega_N^{k}(M)$, then $n(\omega) = 0$. We have
	\begin{align*}
	\chi(\theta^{i_1}\wedge\cdots\wedge\theta^{i_{k}}) = i_{\nu}(\nu^{\flat}\wedge\theta^{i_1}\wedge\cdots\wedge\theta^{i_{k}})  =	\theta^{i_1}\wedge\cdots\wedge\theta^{i_{k}}.
	\end{align*}
By linearity, this implies $\chi\omega = \omega$. On the other hand, if $\omega\in \Omega_N^{k}(M)$, we have $t(\omega) =0$. Then
	\begin{align*}
	\chi(\theta^{i_1}\wedge\cdots\wedge\theta^{i_{k-1}}\wedge\nu^{\flat}) &=  - \nu^{\flat}\wedge i_{\nu}(\theta^{i_1}\wedge\cdots\wedge\theta^{i_{k-1}}\wedge\nu^{\flat})\\
	 &=	-\theta^{i_1}\wedge\cdots\wedge\theta^{i_{k-1}}\wedge\nu^{\flat}.
	\end{align*}
This implies $\chi\omega = -\omega$. Conversely, if $\chi\omega = \omega$, then the above implies
\[	n(\omega) = \chi(n(\omega)) = -n(\omega).\]
Hence $\omega\in \Omega_N^k(M)$. Similarly, $\chi\omega = -\omega$ implies $\omega\in \Omega_D^k(M)$.
\end{proof}

\begin{theorem}[Hodge isomorphism]\label{hodge;theorem}
	Suppose that $(M, g)$ is a compact Riemannian manifold with boundary $\partial M$. Let $f: M\to\R$ be a smooth function. Then
		\[	\mathcal{H}^k_{N,f} (M) \cong H^k(M)\quad\text{and}\quad \mathcal{H}^k_{D,f} (M) \cong H^k(M, \partial M) \cong H^{n-k}(M)\]
		where $H^k(M)$ is the usual $k$-th cohomology class.
\end{theorem}
	
\begin{proof}
	Since $d_f^2 = (d_f^*)^2 = 0$, and $\Omega_D(M)$ is $d_f$-invariant while $\Omega_N(M)$ is $d_f^*$-invariant, we can define two complexes, $(\Omega_D(M),\, d_f)$ and $(\Omega_N(M),\, d_f^*)$, by	
	\begin{align}\label{rel;complex}
			0 \to \Omega_D^0(M) \xrightarrow{d_f} \Omega_D^1(M)\xrightarrow{d_f} \cdots \xrightarrow{d_f} \Omega_D^n(M) \xrightarrow{d_f} 0
	\end{align}
	and
	\begin{align}\label{abs;complex}
		0 \xleftarrow{d_f^*} \Omega_N^0(M) \xleftarrow{d_f^*} \Omega_N^1(M)\xleftarrow{d_f^*} \cdots \xleftarrow{d_f^*} \Omega_N^n(M) \xleftarrow{} 0
	\end{align}
	respectively. The $k$-th cohomology of $(\Omega_N(M),\, d_f)$ is defined to be
	\[	H^k_{f,r}(M) = \frac{\ker d_f|_{\Omega_D^k(M)}}{ \text{im}\, d_f|_{\Omega_D^{k-1}(M)}}, \]
while the $k$-th cohomology of $(\Omega_D(M),\, d_f^*)$ is 
	\[	H^k_{f,a}(M) = \frac{\ker d_f^*|_{\Omega_N^k(M)}}{ \text{im}\, d_f^*|_{\Omega_N^{k+1}(M)}}. \]
Note that the absolute and relative boundary conditions are both elliptic boundary conditions in the classical sense \cite[Example 7.25]{BB12} for the Dirac-type operator $D = d+d^*$. By classical Hodge theory on manifolds with boundary, the complexes (\ref{rel;complex}) and (\ref{abs;complex}) with $f\equiv 0$ are elliptic. Since the differential operators $d$ and $d^*$ agree with $d_f$ and $d^*_f$ in the highest order, respectively, they share the same principal symbols. It follows that $(\Omega_D(M),\, d_f)$ and $(\Omega_N(M),\, d_f^*)$ are both elliptic complexes. Therefore, the Hodge decomposition theorem for elliptic complexes \cite[Theorem 1.5.2]{Gilkey} gives 
		\[\mathcal{H}^k_{N,f} (M) \cong H_{f,a}^k(M)\quad\text{and}\quad \mathcal{H}^k_{D,f} (M)\cong H_{f,r}^{k}(M).\]
		Next, we define a map $\phi_a: H_a^k(M) \to H_{f,a}^k(M)$ by $\phi_a(\omega) = e^{-f}\omega$, and another map $\phi_r: H_r^k(M) \to H_{f,r}^k(M)$ by $\phi_r(\omega) = e^{f}\omega$. Clearly both $\phi_a$ and $\phi_r$ are isomorphisms with inverses $\phi_a^{-1}(\eta) = e^{f}\eta$ and $\phi_r^{-1}(\eta) = e^{-f}\eta$. Here $ H_a^k(M)$ and $ H_r^k(M)$ are the usual absolute and relative cohomologies. Lastly the theorem follows from the fact that $H_r^k(M) \cong H^k(M,\,\partial M) \cong H^{n-k}(M)$ and $H_a^k(M) \cong H^k(M)$. 
	\end{proof}
	
	\bigskip

\section{Integral estimates}\label{section;integral;estimate}

In this section, we will derive integral estimates for two forms in $\Omega^2_D(M)$ and $\Omega^2_N(M)$. Throughout this section, $(M, g)$ is a compact Riemannian manifold with smooth boundary $\partial M$, and $f: M\to\R$ is a smooth function. We first recall the Weitzenb\"ock formula for differential forms:

\begin{lemma}[Weitzenb\"ock formula]\cite[Theorem 7.4.5]{PP}
	Let $\{e_1,\dots, e_n\}$ be an orthonormal frame for $TM$. Then 
	\[	D^2\omega = \Delta_H\omega = -\Delta\omega + \mathcal{R}(\omega),\]
	for any $\omega\in \Omega(M)$, where 
	\[	\mathcal{R}(\omega) = \frac{1}{2}\sum_{i,j=1}^n c(e_i)c(e_j)R(e_i, e_j)\omega.\]
	Here we have used the convention that $R(X,Y)Z = \nabla_X\nabla_YZ - \nabla_Y\nabla_XZ - \nabla_{[X,Y]}Z$, so that $R(e_i, e_j, e_k, e_l) = -g(R(e_i, e_j)e_k,\, e_l)$. 	
\end{lemma}

\begin{lemma}\label{weitzenbock-formula}
	Let $\omega\in \Omega(M)$. Then we have the identity
	
	\begin{align*}
			\langle D_f^2\omega,\, \omega\rangle &= - \langle \Delta\omega,\, \omega\rangle + \langle \mathcal{R}(\omega),\, \omega\rangle + ( |\nabla f|^2 - \Delta f )\, |\omega|^2\\
			&\quad  + 2\sum_{i,j=1}^n (\nabla_{e_i,e_j}^2 f) \langle i_{e_i} \omega,\, i_{e_j} \omega\rangle.
	\end{align*}
	Here $\{e_1,..,e_n\}$ is an orthonormal frame and $\{\theta^1,..,\theta^n\}$ is the dual frame.
\end{lemma}

\begin{proof}
	Let $\{e_1,..,e_n\}$ be an orthonormal frame and $\{\theta^1,..,\theta^n\}$ be the dual frame. Using the Clifford relations (\ref{Clifford;relations}), we compute
	
	\begin{align*}
	D_f^2\omega &= D^2\omega +  \sum_{i,j=1}^n (\nabla_{e_i,e_j}^2 f) c(e_i) \tilde{c}(e_j)\omega + |\nabla f|^2\omega\\
	&= D^2\omega +  \sum_{i,j=1}^n (\nabla_{e_i,e_j}^2 f) c(e_i) (\tilde{c}(e_j) - c(e_j))\omega  + \sum_{i,j=1}^n (\nabla_{e_i,e_j}^2 f) c(e_i) c(e_j)\omega\\
	&\quad  + |\nabla f|^2\omega\\
	&= -\Delta\omega + \mathcal{R}(\omega)  + 2 \sum_{i,j=1}^n (\nabla_{e_i,e_j}^2 f) \theta^i\wedge  (i_{e_j} \omega) + (-\Delta f + |\nabla f|^2)\, \omega.
	\end{align*}
The identity then follows from 
\[	\sum_{i,j=1}^n (\nabla_{e_i,e_j}^2 f) \langle  \theta^i\wedge  (i_{e_j} \omega),\, \omega\rangle = \sum_{i,j=1}^n (\nabla_{e_i,e_j}^2 f) \langle i_{e_i} \omega,\, i_{e_j} \omega\rangle.\]
\end{proof}
 
 To derive integral estimates, we need the following version of Green's identities:
 
 \begin{lemma}\label{green-identity}
	Suppose that $\alpha, \beta\in \Omega(M)$. Then
	\[	-\int_M\langle \Delta\alpha,\, \beta\rangle  = \int_M\langle\nabla\alpha,\, \nabla\beta\rangle  - \int_{\partial M}\langle \nabla_{\nu}\alpha,\, \beta\rangle \]
	and
	\[	\int_M\langle D_f\alpha,\, \beta\rangle  = \int_M\langle\alpha,\, D_f\beta\rangle  + \int_{\partial M} \langle  c(\nu) \alpha,\,\beta\rangle .\]
	Here $\nu$ is the outward unit normal on $\partial M$.
\end{lemma}

\begin{proof}
	Let  $\{e_1,..,e_n\}$ be an orthonormal frame. The first identity follows from applying the divergence theorem on
	\begin{align*}
		-\langle\Delta\alpha,\, \beta\rangle = -\sum_{j=1}^n e_j\big(\langle \nabla_{e_j}\alpha,\,\beta\rangle  \big) + \langle\nabla\alpha,\, \nabla\beta\rangle .
	\end{align*}
Next, since for any $\xi\in TM$, the endomorphism $\tilde{c}(\xi)\in \text{End}(\Omega(M))$ is self-adjoint, and the endomorphism $c(\xi)\in \text{End}(\Omega(M))$ is anti self-adjoint, with respect to the inner product $\langle\cdot, \cdot\rangle$, we have

	\begin{align*}
		\langle D_f\alpha,\, \beta\rangle &= \sum_{j=1}^n\langle c(e_j)\nabla_{e_j}\alpha ,\, \beta\rangle  + \langle \tilde{c}(\nabla f)\alpha,\, \beta\rangle \\
	&=  -\sum_{j=1}^n  e_j\big(\langle \alpha,\, c(e_j)\beta\rangle  \big) + \langle\alpha,\, D\beta\rangle  + \langle \alpha,\, \tilde{c}(\nabla f)\beta\rangle \\
	&=  \sum_{j=1}^n  e_j\big(\langle  c(e_j) \alpha,\,\beta\rangle \big) + \langle\alpha,\, D_f\beta\rangle .
	\end{align*}
Then the second identity also follows from the divergence theorem
\end{proof}

\begin{corollary}\label{integral:identity:without:boundary}
	Let $\omega\in \Omega(M)$. Then we have the integral identity
	
\begin{align*}
	\int_M |D_f\omega|^2 
	&= \int_M |\nabla\omega|^2  + \int_M \langle \mathcal{R}(\omega),\, \omega\rangle  + \int_M(   |\nabla f|^2-\Delta f  )|\omega|^2 \\
	&\quad + \int_M 2\sum_{i,j=1}^n (\nabla_{e_i,e_j}^2 f) \langle i_{e_i} \omega,\, i_{e_j} \omega\rangle\\
	&\quad   - \int_{\partial M}\sum_{i=1}^{n-1} \langle   c(\nu) c(e_i)\nabla_{e_i}\omega,\, \omega\rangle  + \int_{\partial M}\langle \tilde{c}(\nabla f) c(\nu)\omega, \omega\rangle.
\end{align*}
Here $\{e_1,\dots,e_{n-1}\}$ is an orthonormal frame on $T(\partial M)$.
\end{corollary}

\begin{proof}
Combining Lemma \ref{weitzenbock-formula} and Lemma \ref{green-identity}, we have
	
	\begin{align*}
	& \int_M |D_f\omega|^2  + \int_{\partial M} \langle  c(\nu) D_f\omega,\,\omega\rangle \\
	&= \int_M \langle D_f^2\omega,\, \omega\rangle \\
	& = -\int_M\langle \Delta\omega,\, \omega\rangle  + \int_M \langle \mathcal{R}(\omega),\, \omega\rangle   + \int_M(  |\nabla f|^2 -\Delta f )|\omega|^2\\
	&\quad  + \int_M 2\sum_{i,j=1}^n (\nabla_{e_i,e_j}^2 f) \langle i_{e_i} \omega,\, i_{e_j} \omega\rangle \\
	&= \int_M |\nabla\omega|^2  - \int_{\partial M}\langle \nabla_{\nu}\omega,\, \omega\rangle  + \int_M \langle \mathcal{R}(\omega),\, \omega\rangle   \\
	&\quad  + \int_M(  |\nabla f|^2 -\Delta f  )|\omega|^2  + \int_M 2\sum_{i,j=1}^n (\nabla_{e_i,e_j}^2 f) \langle i_{e_i} \omega,\, i_{e_j} \omega\rangle.
	\end{align*}
Hence,

\begin{align*}
	&\int_M |D_f\omega|^2 \\
	&= \int_M |\nabla\omega|^2  + \int_M \langle \mathcal{R}(\omega),\, \omega\rangle   + \int_M(  |\nabla f|^2 -\Delta f  )|\omega|^2 \\
	&\quad + \int_M 2\sum_{i,j=1}^n (\nabla_{e_i,e_j}^2 f) \langle i_{e_i} \omega,\, i_{e_j} \omega\rangle  - \int_{\partial M} \big\langle  c(\nu) D_f\omega + \nabla_{\nu}\omega,\,\omega\big\rangle \\
	&= \int_M |\nabla\omega|^2  + \int_M \langle \mathcal{R}(\omega),\, \omega\rangle  + \int_M(  |\nabla f|^2 -\Delta f )|\omega|^2 \\
	&\quad + \int_M 2\sum_{i,j=1}^n (\nabla_{e_i,e_j}^2 f) \langle i_{e_i} \omega,\, i_{e_j} \omega\rangle\\
	&\quad -\int_{\partial M} \big\langle  c(\nu) D\omega +\nabla_{\nu}\omega,\,\omega\big\rangle  - \int_{\partial M}\langle c(\nu)\tilde{c}(\nabla f)\omega, \omega\rangle .
\end{align*}
The identity then follows from
	\[	 c(\nu) D\omega + \nabla_{\nu}\omega = \sum_{i=1}^{n-1} c(\nu) c(e_i)\nabla_{e_i}\omega \]
and the fact that $\tilde{c}(\nabla f)c(\nu) = -c(\nu)\tilde{c}(\nabla f).$
\end{proof}

It was shown in \cite{MW93} that the curvature term satisfies $\langle \mathcal{R}(\omega), \omega\rangle \geq 0$ for two-forms when the curvature tensor satisfies $R\in C_{PIC}$.  We would like to estimate the curvature term $\langle \mathcal{R}(\omega), \omega\rangle$ from below when $R - \frac{1}{8}\sigma g\owedge g \in C_{PIC}$. We will follow the approach in \cite{MW93}. 

\begin{proposition}\label{estimate;PIC}
	Suppose that $n\geq 4$ is an even integer and the curvature tensor satisfies  $R - \frac{1}{8}\sigma g\owedge g \in C_{PIC}$, then 
		\[	\langle \mathcal{R}\omega,\, \omega\rangle \geq \frac{n-2}{2}\sigma \, |\omega|^2 \]
	for all $\omega\in\Omega^2(M)$.
\end{proposition}

\begin{proof}
	Fix a point $p\in M$ and write $n = 2m$. Following the idea in \cite{MW93}, we can identify $\wedge^2_{\mathbb{C}}T_p^*M$ with $\mathfrak{so}(2n, \mathbb{C})$ by sending $\xi\wedge\eta$ to the skew-symmetric map $L_{\xi\wedge\eta}(X) = \xi(X)\eta^{\sharp} - \eta(X)\xi^{\sharp}$.  Thus we may choose an orthonormal frame $\{e_1,\hdots, e_n\}$ on $T_pM$ with $\{\theta^1,\hdots, \theta^n\}$ the corresponding dual frame, such that 
	\[\omega(p) = \sum_{i=1}^{m}\omega_{i}(p)\, \theta^{2i-1}\wedge\theta^{2i}.\]
	 Using Corollary \ref{curv:term:calc}, we have for $i = 1,\hdots, m$,
		
		\begin{align}\label{PIC;curv;est1}
			&\langle\mathcal{R}(\theta^{2i-1}\wedge\theta^{2i}),\, \theta^{2i-1}\wedge\theta^{2i}\rangle\\
			&\notag = \sum_{k,l=1}^n \Big\{ R_{kl,2i-1, l}\,  \langle \theta^k\wedge \theta^{2i},\, \theta^{2i-1}\wedge\theta^{2i}\rangle - R_{kl,2i, l}\,   \langle \theta^k\wedge \theta^{2i-1},\, \theta^{2i-1}\wedge\theta^{2i}\rangle\\
			&\notag\quad   -  2R_{2i-1, k,2i, l}\,   \langle \theta^k\wedge \theta^l,\, \theta^{2i-1}\wedge\theta^{2i}\rangle \Big\}\\
			&\notag= R_{2i-1, 2i-1} + R_{2i, 2i} - 2R_{2i-1, 2i, 2i-1, 2i}\\
			&\notag= \sum_{j=1,\, j\neq i}^m \Big(R_{2i-1, 2j-1, 2i-1, 2j-1} + R_{2i-1, 2j, 2i-1, 2j} + R_{2i, 2j-1, 2i, 2j-1} + R_{2i, 2j, 2i, 2j} \Big).
		\end{align}
On the other hand, for $j\neq i$,

	\begin{align}\label{PIC;curv;est2}
		&\langle\mathcal{R}(\theta^{2i-1}\wedge\theta^{2i}),\, \theta^{2j-1}\wedge\theta^{2j}\rangle\\
		&\notag= \sum_{k,l=1}^n \Big\{ R_{kl, 2i-1,l}\,  \underbrace{\langle \theta^k\wedge \theta^{2i},\, \theta^{2j-1}\wedge\theta^{2j}\rangle}_{=0} - R_{kl,2i,l}\, \underbrace{\langle   \theta^k\wedge \theta^{2i-1},\, \theta^{2j-1}\wedge\theta^{2j}\rangle}_{=0} \\
		&\notag\quad -  2R_{2i-1,k,2i,l}\,   \langle \theta^k\wedge \theta^l,\, \theta^{2j-1}\wedge\theta^{2j}\rangle\Big\}\\
		&\notag = -2R_{2i-1,2j-1,2i,2j} + 2R_{2i-1,2j,2i,2j-1} \\
		&\notag= 2R_{2i-1, 2i, 2j, 2j-1}.
	\end{align}
The last line in the above follows from the Bianchi identity. Now, combining (\ref{PIC;curv;est1}) and (\ref{PIC;curv;est2}), we have

\begin{align*}
		\langle \mathcal{R}\omega,\, \omega\rangle &= \sum_{i=1}^m |\omega_i|^2 \, \langle\mathcal{R}(\theta^{2i-1}\wedge\theta^{2i}),\, \theta^{2i-1}\wedge\theta^{2i}\rangle + \sum_{i\neq j}^m \omega_i\overline{\omega_j}\, \langle \mathcal{R}(\theta^{2i-1}\wedge\theta^{2i}),\, \theta^{2j-1}\wedge\theta^{2j}\rangle\\
		&\geq \sum_{i=1}^m |\omega_i|^2  \sum_{j\neq i}^m\Big(R_{2i-1, 2j-1, 2i-1, 2j-1} + R_{2i-1, 2j, 2i-1, 2j} + R_{2i, 2j-1, 2i, 2j-1} + R_{2i, 2j, 2i, 2j} \Big)\\
		&\quad - \sum_{i\neq j}^m \frac{1}{2}(|\omega_i|^2 + |\omega_j|^2) |2R_{2i-1, 2i, 2j, 2j-1}|\\
		&\geq \sum_{i=1}^m |\omega_i|^2  \sum_{j\neq i}^m\Big(R_{2i-1, 2j-1, 2i-1, 2j-1} + R_{2i-1, 2j, 2i-1, 2j} + R_{2i, 2j-1, 2i, 2j-1} + R_{2i, 2j, 2i, 2j}\\
		&\quad  - 2|R_{2i-1, 2i, 2j, 2j-1}| \Big)\\
		&\geq \frac{n-2}{2}\sigma\, |\omega|^2
\end{align*}
at the point $p$. The conclusion then follows.
\end{proof}

In the rest of this section, we will derive estimates for the boundary terms in Corollary \ref{integral:identity:without:boundary} under suitable boundary conditions. To that end, we will see that if $\omega\in \mathcal{H}_{N,f}^2(M)$, then the boundary term in the integral formula can be estimated by almost $2$-convexity. On the other hand, if $\omega\in \mathcal{H}_{D,f}^2(M)$, then the boundary term in the integral formula can be estimated by almost $(n-2)$-convexity.

\bigskip

\subsection{Estimates on boundary under almost $2$-convexity}

\begin{lemma}\label{boundary;identity}
Let $\omega\in \Omega(M)$, then we have the following identity on $\partial M$:

\begin{align*}
	&\int_{\partial M}  \sum_{i=1}^{n-1}\langle  c(\nu) c(e_i)\nabla_{e_i}\omega,\, \omega  \rangle \\
	&=  \frac{1}{2}\int_{\partial M}  \sum_{i=1}^{n-1} \langle c(\nu) c(e_i)\nabla_{e_i}\omega,\, \omega - \chi\omega \rangle  + \frac{1}{2}\int_{\partial M} \sum_{i=1}^{n-1}\langle c(\nu) c(e_i)\nabla_{e_i}(\omega - \chi\omega) ,\, \omega \rangle \\
&\quad + \frac{1}{2}\int_{\partial M}H \langle  \omega -  \chi\omega,\, \omega\rangle   -\int_{\partial M} \sum_{i,j=1}^{n-1}A(e_i, e_j)\langle \theta^i\wedge i_{e_j}\omega,\, \omega\rangle .
\end{align*}
Here $\{e_1,\dots,e_{n-1}\}$ is an orthonormal frame on $T(\partial M)$.
\end{lemma}

\begin{proof}
Let $\{e_1,\dots,e_{n-1}\}$ be an orthonormal frame on $T(\partial M)$. For each $i = 1,\dots, n-1$, using the Clifford relations (\ref{Clifford;relations}) we have

\begin{align*}
	&\chi\circ c(\nu) c(e_i)\nabla_{e_i}\omega + c(\nu) c(e_i)\nabla_{e_i}(\chi\omega) \\
	&= -  \tilde{c}(\nu) c(e_i)\nabla_{e_i}\omega \\
	&\quad  + c(\nu) c(e_i) \Big( \tilde{c}(\nabla_{e_i}\nu) c(\nu)\omega + \tilde{c}(\nu) c(\nabla_{e_i}\nu)\omega + \tilde{c}(\nu) c(\nu)\nabla_{e_i}\omega\Big)\\
	&= -\big(\tilde{c}(\nu) c(e_i) + c(e_i)\tilde{c}(\nu)\big)\nabla_{e_i}\omega\\
	&\quad  + \sum_{j=1}^{n-1}A(e_i, e_j)c(\nu) c(e_i) \Big(\tilde{c}(e_j)c(\nu) + \tilde{c}(\nu) c(e_j)\Big)\omega \\
	&= -\sum_{j=1}^{n-1}A(e_i, e_j) c(e_i)\tilde{c}(e_j)\omega + \sum_{j=1}^{n-1}A(e_i, e_j) c(e_i) c(e_j) \chi\omega \\
	&= \sum_{j=1}^{n-1}A(e_i, e_j)c(e_i) ( -\tilde{c}(e_j) + c(e_j))\omega  + \sum_{j=1}^{n-1}A(e_i, e_j) c(e_i) c(e_j)(\chi\omega - \omega).
\end{align*} 
Since $c(e_i)\omega := \theta^i\wedge\omega - i_{e_i}\omega$, $\tilde{c}(e_i)\omega := \theta^i\wedge\omega + i_{e_i}\omega$ and $i_{e_i}i_{e_j} = -i_{e_j}i_{e_i}$, we obtain

\begin{align*}
	\sum_{i,j=1}^{n-1}A(e_i, e_j)c(e_i) ( -\tilde{c}(e_j) + c(e_j))\omega = -2 \sum_{i,j=1}^{n-1}A(e_i, e_j)\theta^i\wedge i_{e_j}\omega.
\end{align*}
On the other hand, 

\begin{align*}
\sum_{i,j=1}^{n-1}A(e_i, e_j) c(e_i) c(e_j)(\chi\omega - \omega) &= -H	(\chi\omega - \omega).
\end{align*}
Finally, using $\chi$ is self-adjoint, we obtain

\begin{align*}
	&\frac{1}{2}\int_{\partial M}  \sum_{i=1}^{n-1} \langle c(\nu) c(e_i)\nabla_{e_i}\omega,\, \omega - \chi\omega \rangle  + \frac{1}{2}\int_{\partial M} \sum_{i=1}^{n-1}\langle c(\nu) c(e_i)\nabla_{e_i}(\omega - \chi\omega) ,\, \omega \rangle  \\
	&= \int_{\partial M}  \sum_{i=1}^{n-1}\langle  c(\nu) c(e_i)\nabla_{e_i}\omega,\, \omega  \rangle  - \frac{1}{2}\int_{\partial M}  \sum_{i=1}^{n-1} \Big\langle \chi\circ   c(\nu) c(e_i)\nabla_{e_i}\omega +  c(\nu) c(e_i)\nabla_{e_i}( \chi\omega),\, \omega  \Big\rangle \\
	&= \int_{\partial M}  \sum_{i=1}^{n-1}\langle  c(\nu) c(e_i)\nabla_{e_i}\omega,\, \omega  \rangle \\
	&\quad + \int_{\partial M} \sum_{i,j=1}^{n-1}A(e_i, e_j)\langle \theta^i\wedge i_{e_j}\omega,\, \omega\rangle  + \frac{1}{2}\int_{\partial M}H \langle \chi\omega - \omega,\, \omega\rangle.
\end{align*}
The desired identity then follows.
\end{proof}

\bigskip

\begin{lemma}\label{estimate;two-convexity}
	Suppose that 
	\[	A(X,X) + A(Y, Y) \geq -\lambda\]
	for all orthonormal $X, Y\in T(\partial M)$ on $\partial M$. Let $\omega\in \Omega_N^2(M)$. Then we have the estimate	
\begin{align*}
\sum_{i,j=1}^{n-1}A(e_i, e_j)\langle \theta^i\wedge i_{e_j}\omega,\, \omega\rangle \geq -\lambda\,|\omega|^2.
\end{align*}
\end{lemma}

\begin{proof}
	Let $\{e_1,\dots,e_{n-1}\}$ be an orthonormal frame on $T(\partial M)$ such that $A(e_i, e_j) = \lambda_i\delta_{ij}$. Then the assumption implies
		\[	\lambda_i + \lambda_j \geq -\lambda\,\quad \forall i,j.\]
	Moreover, let $\{\theta^1,\dots,\theta^{n-1}\}$ be the dual frame. The assumption $\omega\in \Omega_N^2(M)$ implies  $n(\omega) = 0$, hence  we can write $\omega = \sum_{k,l=1}^{n-1}\omega_{kl}\theta^k\wedge \theta^l$ where $\omega_{lk} = -\omega_{kl}$. Using this notation we compute
	
	\begin{align*}
		\theta^i\wedge i_{e_j}\omega &=  \sum_{k,l=1}^{n-1}\omega_{kl}\theta^i \wedge (\delta_j^k \theta^l - \theta^k\delta_j^l)\\
		&= 2\theta^i\wedge\sum_{k=1}^{n-1} \omega_{jk}\theta^k.
	\end{align*}
This gives

	\begin{align*}
	\sum_{i,j=1}^{n-1}A(e_i, e_j)\langle \theta^i\wedge i_{e_j}\omega,\, \omega\rangle 
	&= 2\sum_{i,j,k=1}^{n-1}\lambda_i\delta_{ij} \omega_{jk}\overline{\omega_{ik}}\\
	&= \sum_{i,k=1}^{n-1} (\lambda_i+\lambda_k) |\omega_{ik}|^2\\
	&\geq -\lambda\,|\omega|^2,
	\end{align*}
where the second last line follows from the symmetry $\omega_{lk} = -\omega_{kl}$.
\end{proof}

\bigskip

\begin{corollary}\label{main;estimate;withouPIC}
	Suppose that 
	\[	A(X,X) + A(Y, Y) \geq -\lambda\]
	for all orthonormal $X, Y\in T(\partial M)$ on $\partial M$. Let $\omega\in \mathcal{H}_{N,f}^2(M)$. Then we have the integral estimate

\begin{align*}
0 &\geq \int_M  |\nabla\omega|^2  + \int_M \langle \mathcal{R}(\omega),\, \omega\rangle  + \int_M(  |\nabla f|^2 -\Delta f  )|\omega|^2\\
&\quad  + \int_M 2\sum_{i,j=1}^n (\nabla_{e_i,e_j}^2 f) \langle i_{e_i} \omega,\, i_{e_j} \omega\rangle     + \int_{\partial M} \, \left(\frac{\partial f}{\partial \nu} - \lambda\right)|\omega|^2.
\end{align*}
Here  $\{e_1,\dots,e_{n-1}\}$ is an orthonormal frame on $T(\partial M)$.
\end{corollary}

\begin{proof}
Let $\{e_1,\dots, e_{n-1}\}$ be an orthonormal frame on $T(\partial M)$. Lemma \ref{boundary;correspondence} implies 
	\[	\mathcal{H}^2_{N,f} (M) = \left\{\omega\in \Omega^2(M) |\ D_f\omega = 0,\, \chi\omega = \omega \right\}.\]
 Since $\omega\in \mathcal{H}_{N,f}^2(M)$, for $i = 1,\dots, n-1$, we have
\begin{align*}
		\langle \tilde{c}(e_i)c(\nu)\omega,\, \omega\rangle &= \langle \tilde{c}(e_i)c(\nu)\chi\omega,\, \omega\rangle\\
		&= \langle  c(\nu)\tilde{c}(\nu) \tilde{c}(e_i) c(\nu) \omega,\, \omega\rangle\\
		&= -\langle  \tilde{c}(e_i) c(\nu)\omega,\, \chi\omega\rangle\\
		&= -\langle \tilde{c}(e_i) c(\nu)\omega,\, \omega\rangle.
\end{align*}
Consequently,

\begin{align}\label{boundary;normal;identity1}
	\langle \tilde{c}(\nabla f) c(\nu)\omega, \omega\rangle = \frac{\partial f}{\partial \nu}\, \langle \tilde{c}(\nu) c(\nu)\omega, \omega\rangle  = 	 \frac{\partial f}{\partial \nu}\,|\omega|^2.
\end{align}
Combining (\ref{boundary;normal;identity1}), Corollary \ref{integral:identity:without:boundary}  and Lemma \ref{boundary;identity}, we get
	
\begin{align*}
0 &= \int_M |\nabla\omega|^2  + \int_M \langle \mathcal{R}(\omega),\, \omega\rangle  + \int_M(   |\nabla f|^2 -\Delta f  )|\omega|^2 + \int_M 2\sum_{i,j=1}^n (\nabla_{e_i,e_j}^2 f) \langle i_{e_i} \omega,\, i_{e_j} \omega\rangle \\
	&\quad   + \int_{\partial M} \sum_{i,j=1}^{n-1}A(e_i, e_j)\langle \theta^i\wedge i_{e_j}\omega,\, \omega\rangle    + \int_{\partial M} \, \frac{\partial f}{\partial \nu}|\omega|^2.
\end{align*}
The desired estimate then follows from Lemma \ref{estimate;two-convexity}.
\end{proof}

\bigskip

\subsection{Estimates on boundary under almost $(n-2)$-convexity}

The following Lemma is a variant of Lemma \ref{boundary;identity} for the boundary condition $\chi\omega = -\omega$. 

\begin{lemma}\label{boundary;identity2}
Let $\omega\in \Omega(M)$, then we have the following identity on $\partial M$:

\begin{align*}
	&\int_{\partial M}  \sum_{i=1}^{n-1}\langle  c(\nu) c(e_i)\nabla_{e_i}\omega,\, \omega  \rangle \\
	&=  \frac{1}{2}\int_{\partial M}  \sum_{i=1}^{n-1} \langle c(\nu) c(e_i)\nabla_{e_i}\omega,\, \omega + \chi\omega \rangle  + \frac{1}{2}\int_{\partial M} \sum_{i=1}^{n-1}\langle c(\nu) c(e_i)\nabla_{e_i}(\omega + \chi\omega) ,\, \omega \rangle  \\
&\quad  + \frac{1}{2}\int_{\partial M}H \langle  \omega + \chi\omega ,\, \omega\rangle -\int_{\partial M} \sum_{i,j=1}^{n-1}A(e_i, e_j) \langle i_{e_i}(\theta^j\wedge \omega),\, \omega\rangle   .
\end{align*}
\end{lemma}

\begin{proof}
Let $\{e_1,\dots,e_{n-1}\}$ be an orthonormal frame on $T(\partial M)$. Similar to the proof of Lemma \ref{boundary;identity}, for each $i = 1,\dots, n-1$, we have

\begin{align*}
	&\chi\circ c(\nu) c(e_i)\nabla_{e_i}\omega + c(\nu) c(e_i)\nabla_{e_i}(\chi\omega) \\
	&= \sum_{j=1}^{n-1}A(e_i, e_j)c(e_i) ( -\tilde{c}(e_j) - c(e_j))\omega  + \sum_{j=1}^{n-1}A(e_i, e_j) c(e_i) c(e_j)(\chi\omega + \omega).
\end{align*}
Then

\begin{align*}
	\sum_{i,j=1}^{n-1}A(e_i, e_j)c(e_i) ( -\tilde{c}(e_j) - c(e_j))\omega = 2 \sum_{i,j=1}^{n-1}A(e_i, e_j) i_{e_i}(\theta^j\wedge \omega).
\end{align*}
On the other hand, 

\begin{align*}
\sum_{i,j=1}^{n-1}A(e_i, e_j) c(e_i) c(e_j)(\chi\omega + \omega) &= -H	(\chi\omega + \omega).
\end{align*}
Using the self-adjointness of $\chi$, we obtain

\begin{align*}
	&\frac{1}{2}\int_{\partial M}  \sum_{i=1}^{n-1} \langle c(\nu) c(e_i)\nabla_{e_i}\omega,\, \omega + \chi\omega \rangle  + \frac{1}{2}\int_{\partial M} \sum_{i=1}^{n-1}\langle c(\nu) c(e_i)\nabla_{e_i}(\omega + \chi\omega) ,\, \omega \rangle  \\
	&= \int_{\partial M}  \sum_{i=1}^{n-1}\langle  c(\nu) c(e_i)\nabla_{e_i}\omega,\, \omega  \rangle  + \frac{1}{2}\int_{\partial M}  \sum_{i=1}^{n-1} \Big\langle \chi\circ   c(\nu) c(e_i)\nabla_{e_i}\omega +  c(\nu) c(e_i)\nabla_{e_i}( \chi\omega),\, \omega  \Big\rangle \\
	&= \int_{\partial M}  \sum_{i=1}^{n-1}\langle  c(\nu) c(e_i)\nabla_{e_i}\omega,\, \omega  \rangle \\
	&\quad + \int_{\partial M} \sum_{i,j=1}^{n-1}A(e_i, e_j) \langle i_{e_i}(\theta^j\wedge \omega),\, \omega\rangle  - \frac{1}{2}\int_{\partial M}H \langle \chi\omega + \omega,\, \omega\rangle .
\end{align*}
The desired identity then follows.
\end{proof}

\bigskip

\begin{lemma}\label{estimate;(n-2)-convexity}
	Suppose that 
	\[	A(X_1,X_1) +\cdots + A(X_{n-2}, X_{n-2}) \geq -\lambda\]
	for all orthonormal $X_1,\dots, X_{n-2}\in T(\partial M)$ on $\partial M$. Let $\omega\in \Omega_D^2(M)$. Then we have the estimate	
\begin{align*}
\sum_{i,j=1}^{n-1}A(e_i, e_j) \langle i_{e_i}(\theta^j\wedge \omega),\, \omega\rangle  \geq -\lambda\,|\omega|^2.
\end{align*}
\end{lemma}

\begin{proof}
	Let $\{e_1,\dots,e_{n-1}\}$ be an orthonormal frame on $T(\partial M)$ such that $A(e_i, e_j) = \lambda_i\delta_{ij}$. Then the assumption implies
		\[	\lambda_{i_1} +\cdots + \lambda_{i_{n-2}} \geq -\lambda.\]
	Moreover, let $\{\theta^1,\dots,\theta^{n-1}\}$ be the dual frame. The assumption $\omega\in \Omega_D^2(M)$ implies  $t(\omega) = 0$, hence  we can write $\omega = \sum_{k=1}^{n-1}\omega_{k}\, \theta^k\wedge \nu^{\flat}$. Using this notation we compute
	
	\begin{align*}
		i_{e_i}(\theta^j\wedge \omega) &= \sum_{k=1}^{n-1}\omega_{k}\,  (\delta_i^j\, \theta^k\wedge\nu^{\flat} - \delta_i^k\, \theta^j\wedge\nu^{\flat})\\
		&= \delta_i^j\, \omega  - \omega_i\, \theta^j\wedge\nu^{\flat}.
	\end{align*}
This gives

	\begin{align*}
	\sum_{i,j=1}^{n-1}A(e_i, e_j) \langle i_{e_i}(\theta^j\wedge \omega),\, \omega\rangle 
	&= H |\omega|^2 -  \sum_{i,j=1}^{n-1}\lambda_i\, \delta_{ij}\, \omega_i \overline{\omega_j}\\
	&= \sum_{i=1}^{n-1} (H - \lambda_i) |\omega_{i}|^2\\
	&\geq -\lambda\,|\omega|^2.
	\end{align*}
\end{proof}

\bigskip

\begin{corollary}\label{main;estimate;withouPIC2}
	Suppose that 
	\[	A(X_1,X_1) +\cdots + A(X_{n-2}, X_{n-2}) \geq -\lambda\]
	for all orthonormal $X_1,\dots, X_{n-2}\in T(\partial M)$ on $\partial M$. Let $\omega\in \mathcal{H}_{D,f}^2(M)$. Then we have the integral estimate

\begin{align*}
0 &\geq \int_M  |\nabla\omega|^2  + \int_M \langle \mathcal{R}(\omega),\, \omega\rangle  + \int_M(   |\nabla f|^2-\Delta f  )|\omega|^2\\
&\quad  - \int_M 2\sum_{i,j=1}^n (\nabla_{e_i,e_j}^2 f) \langle i_{e_i} \omega,\, i_{e_j} \omega\rangle + \int_{\partial M} \, \left( -\frac{\partial f}{\partial \nu} - \lambda \right)|\omega|^2.
\end{align*}
Here  $\{e_1,\dots,e_{n-1}\}$ is an orthonormal frame on $T(\partial M)$.
\end{corollary}

\begin{proof}
Lemma \ref{boundary;correspondence} implies 
	\[	\mathcal{H}^2_{D,f} (M) = \left\{\omega\in \Omega^2(M) |\ D_f\omega = 0,\, \chi\omega = -\omega \right\}.\]
Similar to Corollary \ref{main;estimate;withouPIC}, the assumption	$\omega\in \mathcal{H}_{D,f}^2(M)$ implies

\begin{align}\label{boundary;normal;identity2}
	\langle \tilde{c}(\nabla f) c(\nu)\omega, \omega\rangle = \frac{\partial f}{\partial \nu}\, \langle \tilde{c}(\nu) c(\nu)\omega, \omega\rangle  = 	 -\frac{\partial f}{\partial \nu}\,|\omega|^2.
\end{align}
Combining  (\ref{boundary;normal;identity2}), Corollary \ref{integral:identity:without:boundary} and Lemma \ref{boundary;identity2}, we get
	
\begin{align*}
0 &= \int_M  |\nabla\omega|^2  + \int_M \langle \mathcal{R}(\omega),\, \omega\rangle  + \int_M(   |\nabla f|^2 -\Delta f  )|\omega|^2 + \int_M 2\sum_{i,j=1}^n (\nabla_{e_i,e_j}^2 f) \langle i_{e_i} \omega,\, i_{e_j} \omega\rangle \\
	&\quad  + \int_{\partial M} \sum_{i,j=1}^{n-1}A(e_i, e_j) \langle i_{e_i}(\theta^j\wedge \omega),\, \omega\rangle   - \int_{\partial M} \, \frac{\partial f}{\partial \nu}|\omega|^2.
\end{align*}
The desired estimate then follows from Lemma \ref{estimate;(n-2)-convexity}.
\end{proof}

\bigskip

\bigskip

\bigskip

\subsection{Main estimates and Proof of Theorem \ref{main;thm;0}}

\begin{proposition}\label{main;integral;estimate}
	Let $M^n, n\geq 4$ be an even-dimensional compact Riemannian manifold with smooth boundary $\partial M$. Suppose that $(M,g)$ has positive isotropic curvature such that the curvature tensor $R$ satisfies
	\[R - \frac{1}{8}\sigma g\owedge g \in C_{PIC}\]
	for some $\sigma >0$.
Then the followings hold:
	\begin{itemize}
		\item [(i)] Suppose that $A(X,X) + A(Y, Y) \geq - \delta\ $
				for all orthonormal $X, Y\in T(\partial M)$ on $\partial M$. If $b_2(M)\neq 0$, then there exists a nontrivial two-form $\omega\in \mathcal{H}_{N,f}^2(M)$ such that
			\begin{align*}
			0 &\geq \int_M  |\nabla\omega|^2 + \int_M \left( \frac{n-2}{2}\sigma - \Delta f +|\nabla f|^2 \right)|\omega|^2  \\
			&\quad  + \int_M 2\sum_{i,j=1}^n (\nabla_{e_i,e_j}^2 f) \langle i_{e_i} \omega,\, i_{e_j} \omega\rangle + \int_{\partial M} \left( \frac{\partial f}{\partial \nu} - \delta  \right)|\omega|^2.
			\end{align*}
		\item [(ii)] Suppose that $A(X_1,X_1) +\cdots + A(X_{n-2}, X_{n-2}) \geq -\delta\ $
				for all orthonormal $X_1,\dots, X_{n-2}\in T(\partial M)$ on $\partial M$.  If $b_{n-2}(M)\neq 0$, then there exists a nontrivial two-form $\omega\in \mathcal{H}_{D,f}^2(M)$ such that
			\begin{align*}
			0 &\geq  \int_M  |\nabla\omega|^2 + \int_M \left( \frac{n-2}{2}\sigma - \Delta f + |\nabla f|^2 \right)|\omega|^2  \\
			&\quad  + \int_M 2\sum_{i,j=1}^n (\nabla_{e_i,e_j}^2 f) \langle i_{e_i} \omega,\, i_{e_j} \omega\rangle + \int_{\partial M} \left( -\frac{\partial f}{\partial \nu} - \delta  \right)|\omega|^2.
			\end{align*}
	\end{itemize} 
\end{proposition}

\begin{proof}
	The existence of the non-vanishing two forms under the assumed topological conditions follows from the Hodge isomorphism, Theorem \ref{hodge;theorem}. Next, the required estimates follows from combining Proposition \ref{estimate;PIC}, Corollary \ref{main;estimate;withouPIC} and Corollary \ref{main;estimate;withouPIC2}.
\end{proof}

\subsubsection*{Proof of Theorem \ref{main;thm;0}}
Since $M$ is compact, we can find a positive constant $\sigma>0$ such that the isotropic curvature of $(M,g)$ is bounded belowe by $\sigma$. Therefore, Theorem \ref{main;thm;0} follows from taking $f\equiv 0$ and $\delta = 0$ in Proposition \ref{main;integral;estimate}.

\bigskip

\section{Bandwidth estimates}\label{section;bandwidth}

We will prove Theorem \ref{main;thm;1} in this section. We first derive an estimate for the Hessian term appearing in Proposition \ref{main;integral;estimate}.

\begin{lemma}\label{further;estimate;for;bandwidth}\
\begin{itemize}
	\item [(i)] Suppose that $\omega\in \mathcal{H}_{N,f}^2(M)$. Then
	\begin{align*}
	\int_M  \sum_{i,j=1}^n (\nabla_{e_i,e_j}^2 f) \langle i_{e_i} \omega,\, i_{e_j} \omega\rangle \geq 	\ - \int_M \frac{3}{2}|\nabla f|^2 |\omega|^2  -  \int_M \frac{1}{2} |\nabla\omega|^2.
	\end{align*}
	\item [(ii)] Suppose that $\omega\in \mathcal{H}_{D,f}^2(M)$. Then
	\begin{align*}
	&\int_M  \sum_{i,j=1}^n (\nabla_{e_i,e_j}^2 f) \langle i_{e_i} \omega,\, i_{e_j} \omega\rangle \geq 	\int_{\partial M} \frac{\partial f}{\partial \nu}\,|\omega|^2 - \int_M \frac{3}{2}|\nabla f|^2 |\omega|^2  -  \int_M \frac{1}{2} |\nabla\omega|^2.
	\end{align*}
\end{itemize}
\end{lemma}

\begin{proof}
		Let $\{e_1,..,e_n\}$ be an orthonormal frame and $\{\theta^1,..,\theta^n\}$ be the dual frame. We observe that
	
	\begin{align*}
		& \sum_{i,j=1}^n (\nabla_{e_i,e_j}^2 f) \langle i_{e_i} \omega,\, i_{e_j} \omega\rangle \\
		&=\sum_{i,j=1}^n e_i\Big( (\nabla_{e_j}f) \langle i_{e_i} \omega,\, i_{e_j} \omega\rangle\Big)   - \sum_{i,j=1}^n (\nabla_{e_j}f) \langle i_{e_i} \nabla_{e_i}\omega,\, i_{e_j} \omega\rangle - \sum_{i,j=1}^n (\nabla_{e_j}f) \langle i_{e_i} \omega,\, i_{e_j} \nabla_{e_i}\omega\rangle\\
		&= \sum_{i,j=1}^n e_i\Big( (\nabla_{e_j}f) \langle  i_{e_i}\omega,\, i_{e_j} \omega\rangle \Big)  + \langle \delta \omega,\, i_{\nabla f} \omega\rangle - \sum_{i,j=1}^n (\nabla_{e_j}f) \langle  \omega,\, - i_{e_j}(\theta^i\wedge \nabla_{e_i}\omega)  + \delta_j^i\, \nabla_{e_i}\omega \rangle   \\
		&= \sum_{i,j=1}^n e_i\Big( (\nabla_{e_j}f) \langle i_{e_i} \omega,\, i_{e_j} \omega\rangle \Big)  + \langle \delta \omega,\, i_{\nabla f} \omega\rangle +  \langle df\wedge  \omega,\, d\omega \rangle -  \sum_{j=1}^n (\nabla_{e_j}f) \langle  \omega,\,  \nabla_{e_j}\omega \rangle.   \\
	\end{align*}
Since $d_f^*$ is the adjoint of $d_f$ when $\omega\in \mathcal{H}_{N,f}^2(M)$ or $\omega\in \mathcal{H}_{D,f}^2(M)$, we have $0 = d_f\omega = d\omega + df\wedge\omega$ and $0 = d^*_f\omega = d^*\omega + i_{\nabla f}\omega$. This implies

\[	\langle d\omega,\, df\wedge \omega\rangle = - | df\wedge \omega|^2 \quad\text{and}\quad  \langle i_{\nabla f}\omega,\,  d^* \omega \rangle = -| i_{\nabla f}\omega|^2, \]
subsequently

	\[	\langle d\omega,\, df\wedge \omega\rangle +  \langle i_{\nabla f}\omega,\,  d^* \omega \rangle = -|\nabla f|^2 |\omega|^2.\]
Using the divergence theorem, we obtain

	\begin{align*}
	&\int_M \sum_{i,j=1}^n (\nabla_{e_i,e_j}^2 f) \langle i_{e_i} \omega,\, i_{e_j} \omega\rangle\\
	&= \int_{\partial M} \langle i_{\nu} \omega,\, i_{\nabla f} \omega\rangle	- \int_M|\nabla f|^2 |\omega|^2 - \int_M \sum_{j=1}^n (\nabla_{e_j}f) \langle \omega,\,  \nabla_{e_j}\omega \rangle\\
	&\geq  \int_{\partial M} \langle i_{\nu} \omega,\, i_{\nabla f} \omega\rangle  - \int_M \frac{3}{2}|\nabla f|^2 |\omega|^2  -  \int_M \frac{1}{2} |\nabla\omega|^2.
	\end{align*}
	Lastly, the Lemma follows from the fact that $\langle i_{\nu} \omega,\, i_{\nabla f} \omega\rangle = 0$ if $\omega\in \mathcal{H}_{N,f}^2(M)$, and $\langle i_{\nu} \omega,\, i_{\nabla f} \omega\rangle = \frac{\partial f}{\partial \nu}\,|\omega|^2$ if $\omega\in \mathcal{H}_{D,f}^2(M)$.
\end{proof}

\bigskip
\

\subsection{Proof of Theorem \ref{main;thm;1}}
Throughout this subsection, we denote the distance function from the boundary $\partial M$ by
	\[	\rho(x) = \text{dist}_g(x, \partial M).\]
We proceed by assuming, towards a contradiction, that the width of $M$ satisfies

 		\begin{align*}
				\text{dist}_g(\partial_-M,\, \partial_+M) &>  \frac{51}{\sqrt{\sigma}}\, \tan^{-1}\left(\frac{\delta}{\sqrt{\sigma}}\right) := L.
		\end{align*}
Since $\frac{\delta}{\sqrt{\sigma}} < 1$, we use the inequality $\tan^{-1}(y)\geq \frac{\pi}{4}y$ for $y\in [0,1]$ to get
\begin{align}\label{L;inequality}
	L >  \frac{32(n+1)}{\pi(n-3)\sqrt{\sigma}}\, \tan^{-1}\left(\frac{\delta}{\sqrt{\sigma}}\right)	\geq \frac{8(n+1)\delta }{(n-3)\sigma}.
\end{align}

It is possible to find a smooth function $\chi$ such that
\begin{align*}
\begin{cases}
	\chi(x) = -x,&\quad x\in [0,\frac{1}{2}]\\
	0\leq \chi''\leq 4, &\quad x\in [\frac{1}{2},1]\\
	\chi = -1 &\quad x\in [0.9,\infty).
\end{cases}
\end{align*}
The second condition implies that $\chi$ is convex, therefore $\chi' \in [-1,0]$. Now, we define a smooth function $f: M\to\R$ by

\begin{align*}
	f(x) = r \delta \cdot\chi\left(\frac{\rho(x)}{r}\right),
\end{align*}
Here $r := \min\left\{ \frac{L}{2},\  \frac{r_f}{2} \right\}$ and $r_{f}$ is the focal radius of $\partial M$ in $M$. Next, by applying Proposition \ref{main;integral;estimate} and Lemma \ref{further;estimate;for;bandwidth}, we obtain a nontrivial two-form $\omega\in \mathcal{H}_{N,f}^2(M)$ in the case when $b_2(M)\neq 0$, or $\omega\in \mathcal{H}_{D,f}^2(M)$ in the case when $b_{n-2}(M)\neq 0$.  In both cases, the nontrivial two-form $\omega$ satisfies the following estimate:

	\begin{align}\label{bandwidth;main;estimate}
			0 &\geq \int_M \left( \frac{n-2}{2}\sigma - \Delta f - 2|\nabla f|^2 \right)|\omega|^2   + \int_{\partial M} \left( \frac{\partial f}{\partial \nu} - \delta  \right)|\omega|^2\\
			&\notag = \int_M \left( \frac{n-2}{2}\sigma - \delta\, \chi' \left(\frac{\rho}{r}\right)\Delta \rho - \frac{\delta}{r}\,\chi'' \left(\frac{\rho}{r}\right) |\nabla \rho|^2 - 2\delta^2\,\chi' \left(\frac{\rho}{r}\right)^2 |\nabla \rho|^2 \right)|\omega|^2\\
			&\notag\quad + \int_{\partial M} \left( \delta\, \chi' \left(\frac{\rho(0)}{r}\right)\, \langle \nabla\rho(0),\, \nu\rangle  - \delta  \right)|\omega|^2\\
			&\notag\geq \int_M \left( \frac{n-2}{2}\sigma - \delta\, \chi' \left(\frac{\rho}{r}\right)\Delta \rho - \frac{4\delta}{r} - 2\delta^2\right)|\omega|^2.\\\notag 
	\end{align}
Note that  $\chi'=\chi''=0$ when $\rho>r$. Moreover, $\Lambda := \sqrt{\max\left\{-\frac{1}{n-1}Ric_M,\, 0\right\}}$  and so by Theorem \ref{laplace;lower;focal} we have

\begin{align}\label{bandwidth;laplace;bound}
		\Delta \rho \geq  -\frac{(n-1)\Lambda}{\tanh(\frac{r_f \Lambda}{2})}
\end{align}
when $\rho \leq r$. Putting these facts into (\ref{bandwidth;main;estimate}), we obtain

\begin{align}
	0 \geq\, \int_M \left( \frac{n-2}{2}\sigma -  \frac{(n-1)\delta \Lambda}{\tanh( \Lambda r)} - \frac{4\delta}{r} - 2\delta^2\right)|\omega|^2.\\\notag
\end{align}
Now, (\ref{L;inequality}) implies the inequality  $\delta < \frac{(n-3)\sigma L}{8(n+1)}$. It then follows from the assumption $ \delta <  \min\left\{ \frac{(n-2)\sigma}{10(n-1)\Lambda},\, \frac{(n-3)\sigma r_f}{8(n+1)},\, \frac{\sqrt{\sigma}}{2}\right\}$ that
\[	 \delta < \min\left\{ \frac{(n-2)\sigma}{10(n-1)\Lambda},\, \frac{(n-3)\sigma r}{4(n+1)},\, \frac{\sqrt{\sigma}}{2}\right\}.\ \]
Thus, in the case when $\Lambda r\leq 1$, we have $\tanh( \Lambda r) \geq \frac{1}{2}\Lambda r$. This implies

\begin{align*}
		\frac{n-2}{2}\sigma -  \frac{(n-1)\delta \Lambda}{\tanh( \Lambda r)} - \frac{4\delta}{r} - 2\delta^2 &\geq  \frac{n-2}{2}\sigma   - \frac{2(n+1)\delta}{r} - 2\delta^2\\
		&> 0.
\end{align*}
On the other hand, in the case when $\Lambda r > 1$, we have $\tanh( \Lambda r) > \frac{1}{2}$. This implies

\begin{align*}
		\frac{n-2}{2}\sigma -  \frac{(n-1)\delta \Lambda}{\tanh( \Lambda r)} - \frac{4\delta}{r} - 2\delta^2 &\geq  \frac{n-2}{2}\sigma  - 2(n-1)\delta \Lambda  - \frac{4\delta}{r} - 2\delta^2\\
		&\geq  \frac{n-3}{2}\sigma  - \frac{n-3}{5}\sigma   - \frac{n-3}{5}\sigma  \\
		&> 0.
\end{align*}
Therefore we obtain a contradiction.

\bigskip

\section{Focal Radius}\label{section;focal;radius}

In this section we consider focal radius of the boundary hypersurface in manifolds with positive isotropic curvature. We will need the following preparations:

\begin{lemma}\label{focal;hessian;estimate}
	Suppose that $Sec_M \geq 0$ and $H_{\partial M}  \geq - \lambda  $, then
	\begin{align*}
 (\Delta\rho )|\omega|^2 - 2\sum_{i,j=1}^n (\nabla_{e_i,e_j}^2 \rho)\,   \langle \theta^i\wedge \omega,\, \theta^j\wedge \omega\rangle \leq \left(  \frac{(n-1)\lambda}{(n-1) +\lambda\rho} + \frac{4(n-2)}{r_f}\right) |\omega|^2
	\end{align*}
and
	\begin{align*}
 -(\Delta\rho )|\omega|^2 + 2\sum_{i,j=1}^n (\nabla_{e_i,e_j}^2 \rho)\,  \langle i_{e_i} \omega,\, i_{e_j} \omega\rangle  \geq  -\left(  \frac{(n-1)\lambda}{(n-1) +\lambda\rho} + \frac{8}{r_f}\right) |\omega|^2
	\end{align*}	
when $\rho(x)\leq r_f/2$.
\end{lemma}

\begin{proof}
	Firstly, the Laplace comparison Theorem \ref{laplace;comparison} and Corollary \ref{hessian;laplce;lower;focal;positive;curvature} imply that
	\[	 \nabla^2\rho\geq -\frac{2}{r_f}  \quad\text{and}\quad \Delta\rho \leq \frac{(n-1)\lambda}{(n-1) +\lambda\rho}\]
	when $\rho\leq r_f/2$. 

	Now, fix a point $x\in M$ such that $\rho(x)\leq r_f/2$. There exists an orthonormal basis  $\{e_1,\dots, e_{2m}\}$ of $T_xM$ with dual basis $\{\theta^1,\dots, \theta^{2m}\}$ such that
		\[	\omega = \sum_{k=1}^m \omega_k\, \theta^{2k-1}\wedge \theta^{2k},\]
		where $n = 2m$.
	This implies
		\begin{align*}
			&\sum_{i,j=1}^n (\nabla_{e_i,e_j}^2 \rho)\,\langle \theta^i\wedge \omega,\, \theta^j\wedge \omega\rangle \\
			&= \sum_{i,j=1}^n\sum_{k,l=1}^m (\nabla_{e_i,e_j}^2 \rho)\, \omega_k\overline{\omega_l}\, \langle \theta^i\wedge \theta^{2k-1}\wedge \theta^{2k},\, \theta^j\wedge \theta^{2l-1}\wedge \theta^{2l}\rangle\\
			&= \sum_{k=1}^m \sum_{i\neq 2k-1, 2k}^n(\nabla_{e_i,e_j}^2 \rho)\, |\omega_k|^2\\
			&\geq -\frac{2(n-2)}{r_f}\, |\omega|^2.
		\end{align*}
The first inequality then follows.
Next, using the identity
		\[	\sum_{i,j=1}^n (\nabla_{e_i,e_j}^2 \rho) \langle i_{e_i} \omega,\, i_{e_j} \omega\rangle + \sum_{i,j=1}^n (\nabla_{e_i,e_j}^2 \rho) \langle \theta^i\wedge \omega,\, \theta^j\wedge \omega\rangle = (\Delta \rho)|\omega|^2, \]
we have
		\begin{align*}
				\sum_{i,j=1}^n (\nabla_{e_i,e_j}^2 \rho) \langle i_{e_i} \omega,\, i_{e_j} \omega\rangle &=  \sum_{k=1}^m  \big(\nabla_{e_{2k-1}, e_{2k-1}}^2 \rho + \nabla_{e_{2k}, e_{2k}}^2 \rho \big)\, |\omega_k|^2\\
			&\geq -\frac{4}{r_f}\, |\omega|^2.
		\end{align*}
From this the second inequality follows.
\end{proof}

\bigskip

\subsection{Proof of Theorem \ref{thm;PIC;focalradius;b2} in the case when $b_2(M)\neq 0$}\

We proceed by assuming, towards a contradiction, that the focal radius of $\partial M$ satisfies
\begin{align*}
	r_f(\partial M,\, M) > 9\sqrt{\frac{n}{\sigma}}.
\end{align*}
Since $M$ is smooth, we can find positive constants $\lambda$ and $\bar{\lambda}$ such that $H_{\partial M} \geq -\lambda > -\infty$ and $	A(X,X) + A(Y, Y) \geq -\bar{\lambda} > -\infty$
	for all orthonormal $X, Y\in T(\partial M)$ on $\partial M$. Here $H_{\partial M}$ is the boundary mean curvature with respect to the outward normal on $\partial M$. Moreover, we may assume without loss of generality that
\begin{align}
	\lambda > n\sqrt{\sigma} \quad and \quad \bar{\lambda} > n\sqrt{\sigma}.
\end{align}
Note that we do not have any information about $\lambda$ and $\bar{\lambda}$. Throughout this subsection, we denote the distance function from the boundary $\partial M$ by
	\[	\rho(x) = \text{dist}_g(x, \partial M).\] 
Define a function $f(\rho): M\to \R$ by
\begin{definition}\label{def;f;focal}
\begin{align*}
	f(\rho(x)) = \begin{cases}
		-a(\rho(x) - \rho_{\sigma} + \rho_{\lambda} - \bar{\lambda}^{-1})+b , & 0\leq \rho(x) \leq \rho_{\sigma} -\rho_{\lambda} + \bar{\lambda}^{-1}  \\ 
		- 2\log\sin(\beta (\rho (x) - \rho_{\sigma} + \rho_{\lambda})), & \rho_{\sigma} -\rho_{\lambda} +  \bar{\lambda}^{-1}\leq \rho (x) \leq \rho_{\sigma} -\rho_{\lambda} +  \frac{\pi}{2\beta} \\
		0, & \rho(x) \geq \rho_{\sigma} -\rho_{\lambda} + \frac{\pi}{2\beta}.
	\end{cases}
\end{align*}
Here, 
\begin{align*}
	\rho_{\sigma} = (n-1)\sqrt{\frac{15}{n\sigma}},\quad \rho_{\lambda}=\beta^{-1}\arctan\big((1+2\lambda/\beta)^{-1}\big) \in (0,\rho_{\sigma}),  \quad \beta = \sqrt{\frac{(n-2)\sigma}{8}},
\end{align*}
and 
\begin{align*}
	a = 2\beta\cot(\beta\bar{\lambda}^{-1}), \quad  b= -2\log\sin(\beta\bar{\lambda}^{-1}).\\
\end{align*}
\end{definition}
The function $f$ has the following properties:
\begin{enumerate}
	\item $\bar{\lambda}^{-1} < \frac{1}{2}\cdot\frac{\pi}{2\beta}$ and $f$ is well-defined, continuous on $[0,\infty)$.
	\item $f'$ is continuous on $[0,\infty)$. 
	\item $f''$ is bounded and is smooth except at $\rho_{\sigma} -\rho_{\lambda} +  \bar{\lambda}^{-1}$ and $\rho_{\sigma} -\rho_{\lambda} + \frac{\pi}{2\beta}$.
	\item $f$ is a linear function on $[0, \rho_{\sigma} -\rho_{\lambda} + \bar{\lambda}^{-1}]$.
	\item $f' $ is monotone increasing so that $f''\geq 0$. Moreover,  $-a\leq f'\leq 0$. 
\end{enumerate}

\begin{lemma}\label{f;smooth}
	 We have $f \equiv 0$ when $\rho(x) \geq \frac{r_f}{2}$. In particular, $f$ is a smooth function on $M$.
(In fact, we only get that $f$ is twice differentiable with bounded 2nd derivatives, but that is enough for our argument.)
\end{lemma}

\begin{proof}

We first check that $f'(\rho(x))$ is continuous at conjunction points. In the region  $\{\rho(x) \leq \rho_{\sigma} -\rho_{\lambda} + \bar{\lambda}^{-1}\}$, we have
    \[f'(\rho(x)) = -a = -2\beta\cot(\beta\bar{\lambda}^{-1}).\]
In the region  $\{\rho_{\sigma} -\rho_{\lambda} +  \bar{\lambda}^{-1}\leq \rho(x) \leq \rho_{\sigma} -\rho_{\lambda} +  \frac{\pi}{2\beta}\}$, we have
    \[ f'(\rho(x)) = -2\beta \cot (\beta (\rho(x) - \rho_{\sigma} + \rho_{\lambda})).\]
In both cases, we see that $f'(\rho_{\sigma} -\rho_{\lambda} + \bar{\lambda}^{-1}) = -2\beta \cot (\beta\bar{\lambda}^{-1})$.
Thus, $f$ is a smooth function when $\rho(x) < \frac{r_f}{2}$.
On the other hand, we observe that:
\begin{align}
	\rho_{\sigma} -\rho_{\lambda} + \frac{\pi}{2\beta} <& \rho_{\sigma} + \frac{\pi}{2\beta}  = (n-1)\sqrt{\frac{15}{n\sigma}} + \frac{\pi}{2}\sqrt{\frac{8}{(n-2)\sigma}} < \frac{9}{2}\sqrt{\frac{n}{\sigma}} \leq \frac{r_f}{2}.
\end{align}
 Hence when $\rho(x) \geq \frac{r_f}{2}$, the above computation and the definition of $f(\rho)$ implies that $f \equiv 0$. Thus $f$ is smooth everywhere.
\end{proof}

\begin{lemma}\label{focal;boundary;condition;hold}
	We have $\frac{\partial f}{\partial \nu} \geq \bar{\lambda}$ on $\partial M$.
\end{lemma}

\begin{proof}
First observe that on the boundary $\partial M$, the normal derivative $\frac{\partial f}{\partial \nu}$ satisfies	 
\[\frac{\partial f}{\partial \nu}\Big|_{\partial M} = -f'(0) = -2\beta\cot(\beta\bar{\lambda}^{-1}).\]
Since we have assumed that $\bar{\lambda}\geq n\sqrt{\sigma}$, this implies
$ \beta\bar{\lambda}^{-1} \leq \frac{1}{\sqrt{8n}} < 1.$
Using the inequality $2y\cot(y) > 1$ for $y\in [0,1]$, we thus obtain
\[  \bar{\lambda}^{-1}\frac{\partial f}{\partial \nu}\Big|_{\partial M} = 2 \beta\bar{\lambda}^{-1}\cot(\beta\bar{\lambda}^{-1}) \geq 1.\]
\end{proof}

Now, applying Proposition \ref{main;integral;estimate} with the identity
 		\[	\sum_{i,j=1}^n (\nabla_{e_i,e_j}^2 f) \langle i_{e_i} \omega,\, i_{e_j} \omega\rangle + \sum_{i,j=1}^n (\nabla_{e_i,e_j}^2 f) \langle \theta^i\wedge \omega,\, \theta^j\wedge \omega\rangle = (\Delta f)|\omega|^2, \]
we can find a nontrivial two-form $\omega\in \mathcal{H}_{N,f}^2(M)$ such that
			\begin{align*}
			0 &\geq \int_M  |\nabla\omega|^2 + \int_M \left( \frac{n-2}{2}\sigma + \Delta f +|\nabla f|^2 \right)|\omega|^2  \\
			&\quad  - \int_M 2\sum_{i,j=1}^n (\nabla_{e_i,e_j}^2 f) \langle \theta^i\wedge \omega,\, \theta^j\wedge \omega\rangle + \int_{\partial M} \left( \frac{\partial f}{\partial \nu} - \bar{\lambda} \right)|\omega|^2.
			\end{align*} 
Note that Lemma \ref{focal;boundary;condition;hold} implies that the boundary integral in the above integral inequality is nonnegative. Therefore, to reach a contradiction, it suffices to demand 
\begin{align}\label{focal;desired;interior}
	(\Delta f+ |\nabla f|^2)|\omega|^2 - 2 \sum_{i,j=1}^n  (\nabla_{e_i,e_j}^2 f) \langle \theta^i\wedge \omega,\theta^j\wedge \omega\rangle > -\frac{n-2}{2}\sigma |\omega|^2
\end{align}
in $M$. We proceed to the computation

\begin{align}
	&(\Delta f+ |\nabla f|^2)|\omega|^2 -2 \sum_{i,j=1}^n (\nabla_{e_i,e_j}^2 f) \langle \theta^i\wedge \omega,\theta^j\wedge \omega\rangle \nonumber\\
	&=\notag (f'(\rho)\Delta\rho + f'' (\rho)|\nabla \rho|^2  + f'(\rho)^2|\nabla \rho|^2)|\omega|^2 - 2\sum_{i,j=1}^n f'(\rho) (\nabla^2_{e_i,e_j}\rho) \langle \theta^i\wedge\omega, \theta^j\wedge\omega\rangle \\
	&\quad\notag - 2 \sum_{i,j=1}^n f''(\rho) (\nabla_{e_i}\rho) (\nabla_{e_j}\rho) \langle\theta^i\wedge \omega, \theta^j\wedge\omega\rangle \nonumber\\
	&\notag= f'(\rho)(\Delta\rho) |\omega|^2 - 2\sum_{i,j=1}^n f'(\rho) (\nabla^2_{e_i,e_j}\rho) \langle \theta^i\wedge\omega, \theta^j\wedge\omega\rangle\\
	&\quad\notag + ( f''(\rho)   + f'(\rho)^2)|\omega|^2- 2f''(\rho) |d\rho \wedge \omega|^2 \nonumber\\
	&\geq  f'(\rho)(\Delta\rho) |\omega|^2 - 2\sum_{i,j=1}^n f'(\rho) (\nabla^2_{e_i,e_j}\rho) \langle \theta^i\wedge\omega, \theta^j\wedge\omega\rangle + (- f''(\rho)   + f'(\rho)^2)|\omega|^2\nonumber\\
	&\geq \left( \frac{(n-1)\lambda f'(\rho)}{(n-1)+\lambda\rho} + \frac{4(n-2)f'(\rho)}{r_f}- f''(\rho)   + f'(\rho)^2\right)|\omega|^2.\nonumber
\end{align} 
Here, the last inequality in the above follows from Lemma \ref{focal;hessian;estimate}.
Hence, to verify (\ref{focal;desired;interior}), it is equivalent to verify the inequality

\begin{align}\label{focal;desired;inequality;interior}
	\left(\frac{(n-1)\lambda }{(n-1) + \lambda \rho}  + \frac{4(n-2)}{r_f}\right)f'(\rho) - f''(\rho) +f'(\rho)^2 > -\frac{(n-2)\sigma}{2}.
\end{align}

To proceed, we first check that in the region $\{\rho_{\sigma} -\rho_{\lambda} +  \frac{\pi}{2\beta} \geq\rho(x) \geq \rho_{\sigma} -\rho_{\lambda} +  \bar{\lambda}^{-1}\}$, we have

\begin{align}\label{focal;positive;Ric;f'sq-f''}
	-f''(\rho) + \frac{f'(\rho)^2}{2} =& -2\beta^2\csc^2(\beta (\rho- \rho_{\sigma} + \rho_{\lambda})) + 2\beta^2\cot^2(\beta (\rho- \rho_{\sigma} + \rho_{\lambda})) \nonumber\\
	=& -2\beta^2 \nonumber\\
	=& -\frac{(n-2)\sigma}{4}.
\end{align}
And in the regions $\{\rho(x) < \rho_{\sigma} -\rho_{\lambda} +  \bar{\lambda}^{-1}\}$ and $\{\rho (x) > \rho_{\sigma} -\rho_{\lambda} +  \frac{\pi}{2\beta}\}$,  $f$ is a linear function, so $-f''(\rho) + \frac{f'(\rho)^2}{2} \geq 0$. 

\bigskip
We next show that (\ref{focal;desired;inequality;interior}) always holds.

\noindent\textbf{In the region $\left\{\rho(x) > \frac{r_f}{2}\right\}$ :}.
In this region we have $f \equiv 0$, so clearly (\ref{focal;desired;inequality;interior}) holds.

\noindent\textbf{In the region $\left\{\frac{r_f}{2} > \rho (x) \geq \rho_{\sigma} = (n-1)\sqrt{\frac{15}{n\sigma}}\right\}$ :}
In this region, we find that
\begin{align}
	0 <& \frac{(n-1)\lambda }{(n-1) + \lambda \rho}  + \frac{4(n-2)}{r_f} <  \frac{n-1}{\rho_{\sigma}} + \frac{4(n-2)}{r_f} 
	\leq \sqrt{\frac{n\sigma}{15}} + \frac{4(n-2)\sqrt{\sigma}}{9\sqrt{n}} \leq \sqrt{\frac{(n-2)\sigma}{2}}.
\end{align}
Since $\frac{x^2}{2} + bx \geq -\frac{b^2}{2}$ we have
\begin{align}\label{focal;positive;Ric;f'sq+f'}
	\left(\frac{(n-1)\lambda }{(n-1) + \lambda \rho}  + \frac{4(n-2)}{r_f}\right)f'(\rho) + \frac{f'(\rho)^2}{2} > -\frac{(n-2)\sigma}{4}.
\end{align}
Consequently (\ref{focal;positive;Ric;f'sq-f''}) and (\ref{focal;positive;Ric;f'sq+f'}) imply that (\ref{focal;desired;inequality;interior}) holds in this region. 

\noindent\textbf{In the region $\left\{\rho(x) \leq \rho_{\sigma} = (n-1)\sqrt{\frac{15}{n\sigma}}\right\}$ :} In this region, we find that
\begin{align}
	0 < \frac{(n-1)\lambda }{(n-1) + \lambda \rho}  + \frac{4(n-2)}{r_f} <  \lambda +  \sqrt{\frac{(n-2)\sigma}{4}} < 2\lambda
\end{align}
where in the last inequality we used the assumption that $\lambda > n\sqrt{\sigma}$.

Since $ -\rho_{\lambda} +  \frac{\pi}{2\beta} = \beta^{-1}(\frac{\pi}{2} - \arctan\big((1+2\lambda/\beta)^{-1}\big))> 0$ and  we may take $\bar{\lambda}$ large enough so that $\rho_{\sigma} \geq \rho_{\sigma} -\rho_{\lambda} +  \bar{\lambda}^{-1}$, we have $\rho_{\sigma}\in (\rho_{\sigma} -\rho_{\lambda} +  \bar{\lambda}^{-1},\, \rho_{\sigma} -\rho_{\lambda}  +  \frac{\pi}{2\beta})$. Therefore,
 
\begin{align}
	f'(\rho) \leq f'(\rho_{\sigma}) = -2\beta\cot(\beta \rho_{\lambda}) = -2(\beta+ 2\lambda) =  -4\lambda - \sqrt{\frac{(n-2)\sigma}{2}} < -4\lambda.
\end{align}

Hence 
\begin{align}\label{focal;positive;Ric;f'sq+f';v2}
	\left(\frac{(n-1)\lambda }{(n-1) + \lambda \rho}  + \frac{4(n-2)}{r_f}\right)f'(\rho) + \frac{f'(\rho)^2}{2} > 0.
\end{align} 
Therefore (\ref{focal;desired;inequality;interior}) holds.

To summarize, we proved that if $r_f > 9\sqrt{\frac{n}{\sigma}}$, then $f$ is a smooth function and  (\ref{focal;desired;inequality;interior})  always holds, this leads to a contradiction.

\bigskip

\subsection{Proof of Theorem \ref{thm;PIC;focalradius;b2} in the case when $b_{n-2}(M)\neq 0$}\

The argument is similar to the one in the case when $b_{2}(M)\neq 0$. Again, we assume in contrary that the focal radius of $\partial M$ satisfies
\begin{align*}
	r_f(\partial M,\, M) > 9\sqrt{\frac{n}{\sigma}}.
\end{align*}
Since $M$ is smooth, we can find positive constants $\lambda$ and $\bar{\lambda}$ such that $H_{\partial M} \geq -\lambda > -\infty$ and $	A(X_1,X_1) +\cdots + A(X_{n-2}, X_{n-2}) \geq -\bar{\lambda} > -\infty$
	for all orthonormal $X_1,\hdots, X_{n-2}\in T(\partial M)$ on $\partial M$.  Moreover, we may assume without loss of generality that
\begin{align}
	\lambda > n\sqrt{\sigma} \quad and\quad \bar{\lambda} > n\sqrt{\sigma}.
\end{align}
We denote the distance function from the boundary $\partial M$ by
	\[	\rho(x) = \text{dist}_g(x, \partial M).\] 
Similarly, we define a function $f:M\to \R$ by
\begin{definition}\label{def;f;focal;bn-2}
\begin{align*}
	f(\rho(x)) = \begin{cases}
		a(\rho(x) - \rho_{\sigma} + \rho_{\lambda} - \bar{\lambda}^{-1}) - b , & 0\leq \rho(x) \leq \rho_{\sigma} -\rho_{\lambda} + \bar{\lambda}^{-1}  \\ 
		 2\log\sin(\beta (\rho (x) - \rho_{\sigma} + \rho_{\lambda})), & \rho_{\sigma} -\rho_{\lambda} +  \bar{\lambda}^{-1}\leq \rho (x) \leq \rho_{\sigma} -\rho_{\lambda} +  \frac{\pi}{2\beta} \\
		0, & \rho(x) \geq \rho_{\sigma} -\rho_{\lambda} + \frac{\pi}{2\beta}.
	\end{cases}
\end{align*}
Here, 
\begin{align*}
	\rho_{\sigma} = (n-1)\sqrt{\frac{15}{n\sigma}},\quad \rho_{\lambda}=\beta^{-1}\arctan\big((1+2\lambda/\beta)^{-1}\big) \in (0,\rho_{\sigma}),  \quad \beta = \sqrt{\frac{(n-2)\sigma}{8}},
\end{align*}
and 
\begin{align*}
	a = 2\beta\cot(\beta\bar{\lambda}^{-1}), \quad  b= -2\log\sin(\beta\bar{\lambda}^{-1}).\\
\end{align*}
\end{definition}
The function $f$ has the following properties:
\begin{enumerate}
	\item $\bar{\lambda}^{-1} < \frac{1}{2}\cdot\frac{\pi}{2\beta}$ and $f$ is well-defined, continuous on $[0,\infty)$.
	\item $f'$ is continuous on $[0,\infty)$. 
	\item $f''$ is bounded and is smooth except at $\rho_{\sigma} -\rho_{\lambda} +  \bar{\lambda}^{-1}$ and $\rho_{\sigma} -\rho_{\lambda} + \frac{\pi}{2\beta}$.
	\item $f$ is a linear function on $[0, \rho_{\sigma} -\rho_{\lambda} + \bar{\lambda}^{-1}]$.
	\item $f' $ is monotone decreasing so that $f''\leq 0$. Moreover,  $a\geq f'\geq 0$. 
\end{enumerate}
By replacing $f$ with $-f$ in Lemma \ref{f;smooth} and Lemma \ref{focal;boundary;condition;hold}, we obtain

\begin{lemma}\label{focal;boundary;condition;hold;bn-2}
		 We have $f \equiv 0$ when $\rho(x) \geq \frac{r_f}{2}$. In particular, $f$ is a smooth function on $M$. Moreover, we have $-\frac{\partial f}{\partial \nu} \geq \bar{\lambda}$ on $\partial M$.
\end{lemma}
Then, applying Proposition \ref{main;integral;estimate} with Lemma \ref{focal;boundary;condition;hold;bn-2}, 
we can find a nontrivial two-form $\omega\in \mathcal{H}_{D,f}^2(M)$ such that

	\begin{align*}
			0 &\geq  \int_M  |\nabla\omega|^2 + \int_M \left( \frac{n-2}{2}\sigma - \Delta f + |\nabla f|^2 \right)|\omega|^2    + \int_M 2\sum_{i,j=1}^n (\nabla_{e_i,e_j}^2 f) \langle i_{e_i} \omega,\, i_{e_j} \omega\rangle. 
	\end{align*}
Therefore, to reach a contradiction, it suffices to demand 

\begin{align}\label{focal;desired;interior;bn-2}
	(-\Delta f+ |\nabla f|^2)|\omega|^2 + 2 \sum_{i,j=1}^n (\nabla_{e_i,e_j}^2 f) \langle i_{e_i} \omega,\, i_{e_j} \omega\rangle > -\frac{n-2}{2}\sigma |\omega|^2
\end{align}
in $M$. We proceed to the computation

\begin{align}
	&(-\Delta f+ |\nabla f|^2)|\omega|^2 + 2 \sum_{i,j=1}^n (\nabla_{e_i,e_j}^2 f) \langle i_{e_i} \omega,\, i_{e_j} \omega\rangle \nonumber\\
	&=\notag ( -f'(\rho)\Delta\rho - f'' (\rho)|\nabla \rho|^2  + f'(\rho)^2|\nabla \rho|^2)|\omega|^2 + 2\sum_{i,j=1}^n f'(\rho) (\nabla^2_{e_i,e_j}\rho) \langle i_{e_i} \omega,\, i_{e_j} \omega\rangle \\
	&\quad\notag + 2 \sum_{i,j=1}^n f''(\rho) (\nabla_{e_i}\rho) (\nabla_{e_j}\rho) \langle i_{e_i} \omega,\, i_{e_j} \omega\rangle \nonumber\\
	&\notag= -f'(\rho)(\Delta\rho) |\omega|^2 + 2\sum_{i,j=1}^n f'(\rho) (\nabla^2_{e_i,e_j}\rho) \langle i_{e_i} \omega,\, i_{e_j} \omega\rangle + ( - f''(\rho)   + f'(\rho)^2)|\omega|^2\\
	&\quad\notag  + 2f''(\rho) |i_{\nabla\rho} \omega|^2 \nonumber\\
	&\geq  - f'(\rho)(\Delta\rho) |\omega|^2 + 2\sum_{i,j=1}^n f'(\rho) (\nabla^2_{e_i,e_j}\rho) \langle i_{e_i} \omega,\, i_{e_j} \omega\rangle + (f''(\rho)   + f'(\rho)^2)|\omega|^2\nonumber\\
	&\geq \left( - \frac{(n-1)\lambda f'(\rho)}{(n-1) +\lambda\rho} - \frac{8 f'(\rho)}{r_f} + f''(\rho)   + f'(\rho)^2\right)|\omega|^2.\nonumber
\end{align} 
Here, the last inequality in the above follows from Lemma \ref{focal;hessian;estimate}.
Hence, to verify (\ref{focal;desired;interior;bn-2}), it is equivalent to verify the inequality

\begin{align}\label{focal;desired;inequality;interior;bn-2}
	-\left(\frac{(n-1)\lambda }{(n-1) +\lambda\rho} + \frac{8 }{r_f}\right)f'(\rho) + f''(\rho) +f'(\rho)^2 > -\frac{(n-2)\sigma}{2}.
\end{align}
To proceed, we first check that in the region $\{\rho_{\sigma} -\rho_{\lambda} +  \frac{\pi}{2\beta} \geq\rho(x) \geq \rho_{\sigma} -\rho_{\lambda} +  \bar{\lambda}^{-1}\}$, we have

\begin{align}\label{focal;positive;Ric;f'sq-f'';bn-2}
	f''(\rho) + \frac{f'(\rho)^2}{2} =& -2\beta^2\csc^2(\beta (\rho- \rho_{\sigma} + \rho_{\lambda})) + 2\beta^2\cot^2(\beta (\rho- \rho_{\sigma} + \rho_{\lambda})) \nonumber\\
	=& -2\beta^2 \nonumber\\
	=& -\frac{(n-2)\sigma}{4}.
\end{align}
And in the regions $\{\rho(x) < \rho_{\sigma} -\rho_{\lambda} +  \bar{\lambda}^{-1}\}$ and $\{\rho (x) > \rho_{\sigma} -\rho_{\lambda} +  \frac{\pi}{2\beta}\}$,  $f$ is a linear function, so $f''(\rho) + \frac{f'(\rho)^2}{2} \geq 0$. 

\bigskip

We next show that (\ref{focal;desired;inequality;interior;bn-2}) always holds.

\noindent\textbf{In the region $\left\{\rho(x) > \frac{r_f}{2}\right\}$ :}.
In this region we have $f \equiv 0$, so clearly (\ref{focal;desired;inequality;interior;bn-2}) holds.

\noindent\textbf{In the region $\left\{\frac{r_f}{2} > \rho (x) \geq \rho_{\sigma} = (n-1)\sqrt{\frac{15}{n\sigma}}\right\}$ :}
In this region, we find that
\begin{align}
	0 < \frac{(n-1)\lambda }{(n-1) + \lambda \rho}  + \frac{8}{r_f} <  \frac{n-1}{\rho_{\sigma}} + \frac{8}{r_f} \leq \sqrt{ \frac{n\sigma}{15}} + \frac{8}{9}\sqrt{\frac{\sigma}{n}} \leq \sqrt{\frac{(n-2)\sigma}{2}}.
\end{align}
Since $\frac{x^2}{2} + bx \geq -\frac{b^2}{2}$ we have
\begin{align}\label{focal;positive;Ric;f'sq+f';bn-2}
	-\left(\frac{(n-1)\lambda }{(n-1) + \lambda \rho}  + \frac{8}{r_f}\right)f'(\rho) + \frac{f'(\rho)^2}{2} > -\frac{(n-2)\sigma}{4}.
\end{align}
Consequently (\ref{focal;positive;Ric;f'sq-f'';bn-2}) and (\ref{focal;positive;Ric;f'sq+f';bn-2}) imply that (\ref{focal;desired;inequality;interior;bn-2}) holds in this region. 

\noindent\textbf{In the region $\left\{\rho(x) \leq \rho_{\sigma} = (n-1)\sqrt{\frac{15}{n\sigma}}\right\}$ :} In this region, we find that
\begin{align}
	0 < \frac{(n-1)\lambda }{(n-1) + \lambda \rho}  + \frac{8}{r_f} <  \lambda +  \sqrt{\sigma} < 2\lambda.
\end{align}
Since $ -\rho_{\lambda} +  \frac{\pi}{2\beta} = \beta^{-1}(\frac{\pi}{2} - \arctan\big((1+2\lambda/\beta)^{-1}\big))> 0$ and  we may take $\bar{\lambda}$ large enough so that $\rho_{\sigma} \geq \rho_{\sigma} -\rho_{\lambda} +  \bar{\lambda}^{-1}$, we have $\rho_{\sigma}\in (\rho_{\sigma} -\rho_{\lambda} +  \bar{\lambda}^{-1},\, \rho_{\sigma} -\rho_{\lambda}  +  \frac{\pi}{2\beta})$. Therefore,
 
\begin{align}
	-f'(\rho) \leq -f'(\rho_{\sigma}) = -2\beta\cot(\beta \rho_{\lambda}) = -2(\beta+ 2\lambda) =  -4\lambda - \sqrt{\frac{(n-2)\sigma}{2}} < -4\lambda.
\end{align}
Hence 
\begin{align}\label{focal;positive;Ric;f'sq+f';v2;bn-2}
	-\left(\frac{(n-1)\lambda }{(n-1) + \lambda \rho}  + \frac{8}{r_f}\right)f'(\rho) + \frac{f'(\rho)^2}{2} > 0.
\end{align} 
Therefore (\ref{focal;desired;inequality;interior;bn-2}) always holds, and this completes the proof.

\bigskip

\subsection{Proof of Theorem \ref{thm;PIC;focalradius;hypersurface;b2}}
Again, we argue by contradiction. Suppose in contrary that 
\begin{align*}
	r_f(\Sigma,\, M) > 9\sqrt{\frac{n}{\sigma}} =: L.
\end{align*}
We can find a normal neighborhood $V\cong \Sigma\times [0,1]$ of $\Sigma$ such that
		\[	\partial V = \Sigma \sqcup \tilde{\Sigma}\]
		with $\tilde{\Sigma} = \{\exp_x(-L\, \nu(x)): x\in\Sigma\}$. Here $\nu$ is the outward unit normal along $\Sigma$ with respect to $V$. The definition of $\tilde{\Sigma}$ implies that $r_f(\tilde{\Sigma},\, V)> L$, by the uniqueness of normal geodesics within $V$. This leads to $r_f(\partial V,\, V)> L$. On the other hand, if $b_2(\Sigma)\neq 0$, then $b_2(V)\neq 0$. Similarly, if $b_{n-2}(\Sigma)\neq 0$, then $b_{n-2}(V)\neq 0$. The proof then proceeds analogously to the argument in the proof of Theorem \ref{thm;PIC;focalradius;b2}.

\bigskip
\

\section{A counter example}\label{section;counter}
In this section, we aim to find a manifold with nonzero $b_2$ or $b_{n-2}$,  uniformly PIC, but with arbitrarily large bandwidth. Indeed, we prove a slightly more general statement as follows:
\begin{proposition}\label{mainl;counter,example}
Given $2\leq k\leq n-2$, there exists $C = C(n,k)$ with the following property.  For any $\sigma >0$ and $L > \frac{2}{\sqrt{\sigma}}$, there exists a compact Riemannian band $M^n$ with boundary $\partial M = \partial_-M\sqcup\partial_+M$ such that following conditions hold:
	\begin{itemize}
		\item [(i)] The curvature tensor satisfies $R - \frac{1}{8}\sigma g\owedge g \in C_{PIC}$;
		\item[(ii)]  $b_k(M)\neq 0$. 
            \item[(iii)] The boundary curvature  $|A| \leq C\sqrt{\sigma}$.
            \item[(iv)] The width $\text{dist}_g(\partial_-M,\, \partial_+M) > L$.
	\end{itemize} 
\end{proposition}

\begin{remark}
    This example shows that the boundary condition in the theorem \ref{main;thm;1} is necessary. Specifically, if we take $\delta = kC\sqrt{\sigma}$ in Theorem \ref{main;thm;1}, where $k = 2$ or $k = n-2$, i.e., 
\[A(X_1,X_1) +\cdots + A(X_{k}, X_{k}) \geq - kC\sqrt{\sigma} \]
for all orthonormal $X_1,\hdots, X_{k}\in T(\partial M)$, then the example constructed in Propostiion \ref{mainl;counter,example} will satisfy the boundary condition and topological assumption in Theorem \ref{main;thm;1}, but will defy the bandwidth estimate. In other words, there must be an upper bound for $\delta$ for Theorem \ref{main;thm;1} to hold.
\end{remark}

We start with some conventions and notations:
\begin{itemize}
    \item $A \sqcup B$ means disjoint union of $A,B$, i.e., $\bar{A}\cap \bar{B} = \phi$. 
    \item $\cong$ denotes diffeomorphism up to the boundary or isomorphism of cohomology groups, depending on the context.
    \item  $\simeq$ denotes homotopy equivalence.
    \item A smooth domain means that the boundary of the domain is a smooth manifold.
\end{itemize}  
We first prove a topological lemma:

\begin{lemma}\label{topology;punctuation}
	 For all $n\geq 4$,  $2\leq k \leq n-2$, and a compact manifold $M^n$ (possibly with boundary), puncturing a disk preserves the $k$-th de Rham cohomology:
\begin{align*}
    H^k(M^n\backslash B^{n} ) \cong H^k(M^n) \
\end{align*}
where $B^{n}\subset M^n$ is diffeomorphic to a solid ball in $\mathbb{R}^{n}$ up to boundary.
\end{lemma}
\begin{proof}
By our assumption, $\overline{M^n\backslash B^n} \cap \overline{B^n}  \simeq S^{n-1}$. 
By the Mayer-Vietoris sequence,
\begin{align*}
    H^{k-1}(S^{n-1})\rightarrow& H^k(M) \rightarrow
    H^k(M\backslash B^n) \oplus H^k(B^{n})\rightarrow H^{k}(S^{n-1})
\end{align*}
is an exact sequence.

Since $H^{k}(S^{n-1}) = H^{k-1}(S^{n-1}) = 0$ by our assumption, we have
\begin{align*}
    H^k(M) \cong H^k(M\backslash B^n).
\end{align*}
\end{proof}

As a consequence, if we use $(S^{n-1}\times S^1)\backslash (B^{n} \sqcup B^{n})$ to denote puncturing two disjoint balls from $S^{n-1}\times S^1$, then
\begin{align*}
	H^k((S^{n-1}\times S^1)\backslash (B^{n} \sqcup B^{n})) = 0
\end{align*}
for $n\geq 4$ and $2\leq k \leq n-2$. 
\bigskip

Next, we record a lemma about finding product-embedded hypersurface:

\begin{lemma}\label{existence of product hypersurface}
Suppose that $M^{n-k}\subset \mathbb{R}^{n-k+1}$ is a smooth embedded closed hypersurface and $B^k$ is a round solid ball in $\mathbb{R}^k$. Then
    \begin{enumerate}
        \item There exists an open domain $\Omega\subset \mathbb{R}^{n}$ that is diffeomorphic to $B^{k}\times M^{n-k}$ up to the boundary.
        \item Moreover, $\partial \Omega \cong S^{k-1} \times M^{n-k}$. 
        \item Furthermore, if $M$ and $N$ are connected, then $\mathbb{R}^{n}\backslash \Omega$ is connected.
    \end{enumerate}	
\end{lemma}
\begin{proof}
    Since $M^{n-k}\subset \mathbb{R}^{n-k+1} \subset \mathbb{R}^{n}$, we view $M$ as being embedded in $\mathbb{R}^n$, so $\mathbb{R}^n\backslash M$ is path connected. Since $M$ is smoothly embedded, there exists $\varepsilon>0$ such that the $\varepsilon$-tubular neighborhood of $M$: 
\begin{align*}
    M_{\varepsilon} = \{x \in \mathbb{R}^n : d(x,M) < \varepsilon\}
\end{align*}
is diffeomorphic to $B^k\times M^{n-k}$, $\partial M_{\varepsilon} \cong S^{k-1}\times M^{n-k}$, and $\mathbb{R}^n\backslash M_{\varepsilon}$ is connected. 
\end{proof}

In the following lemma, we construct a nice connected subset of $\mathbb{R}^n$ with nontrivial topology, serving as a building block for the Proposition \ref{mainl;counter,example}.
 \begin{lemma}\label{counter;example;nontrivial,block}
    For $n\geq 4$ and $2\leq k\leq n-2$, there exists a smooth open domain $\Omega^{n } \subset \mathbb{R}^{n}$ such that $\partial \Omega$ and $\Omega^c$ are connected, and $b_{k}(\Omega^c) = 1$. 
\end{lemma}
\begin{proof}
    Let $B^m\subset \mathbb{R}^m$ denote a solid round ball.
By proposition \ref{existence of product hypersurface}, there exists $\Omega^{n} \subset \mathbb{R}^{n}$ with the following properties:
\begin{enumerate}
    \item $\Omega \cong B^{k+1}\times S^{n-k-1}$,
    \item $\partial \Omega \cong S^{k}\times S^{n-k-1}$,
    \item $\Omega^c = \mathbb{R}^{n}\backslash \Omega$ is connected.
\end{enumerate}

Clearly, we have 
\begin{align}
    H^{k}(\Omega \cup \Omega^c) = H^{k+1}(\Omega \cup \Omega^c) = 0.
\end{align}
Since $\overline{\Omega} \cap \overline{\Omega^c} \simeq \partial \Omega$, by the Mayer-Vietoris sequence we conclude that 
\begin{align*}
    H^{k}(\Omega) \oplus H^{k}(\Omega^c) \cong  H^{k}(\partial \Omega). 
\end{align*}
Thus, 
\begin{align*}
    b_{k}(\Omega^c) = b_{k}(\partial \Omega) - b_{k}(\Omega).
\end{align*}
Plugging in the construction, we have
\begin{align*}
    b_k(\Omega^c) =& b_k(S^k\times S^{n-k-1}) - b_k(B^{k+1} \times S^{n-k-1}) \\
                =&  b_k(S^k\times S^{n-k-1}) - b_k(S^{n-k-1})  \quad \text{(homotopy invariance)}\\
                =& 1.
\end{align*}
The last equality follows from the following fact:
\begin{align*}
    b_k(S^k\times S^{n-k-1}) = \begin{cases}
      2, & \text{if}\ n = 2k +1 \\
      1, & \text{if}\ n \neq 2k +1
    \end{cases},
\end{align*}
\begin{align*}
    b_k(S^{n-k-1}) = \begin{cases}
      1, & \text{if}\ n = 2k + 1\\
      0, & \text{if}\ n \neq 2k + 1
    \end{cases}.
\end{align*}

\end{proof}

\begin{remark}
Lemma \ref{counter;example;nontrivial,block} still holds if we interpret $\Omega^c = B \backslash \Omega$ for some ball $B \supset \overline{\Omega}$. 
\end{remark}

We can now prove Proposition \ref{mainl;counter,example}: 

\begin{proof}[Proof of Proposition \ref{mainl;counter,example}]

After a scaling, we may assume that $\sigma = 1$ and $L  > 2$. 

First, let $M_1 = S^{n-1}\times S^1$ and consider standard coordinates $(\Theta, r)$, where $\Theta \in S^n$ and $r\in [0,4L]$. (so $(\Theta, 0)$ and $(\Theta, 4L)$ are identified as the same point.)

We endow $M_1$ with a warped product metric $g_{S^{n-1}} + dr^2$, where $g_{S^{n-1}}$ is the metric of the round $S^{n-1}$ of radius 1. 
Fix any $\Theta \in S^{n-1}$ and consider two points $p_1 = (\Theta, 0)$, $p_2 = (\Theta, 2L)$. 
Let $B_i = B_g(p_i, 1)\subset M_1$ ($i=1,2$) be two disjoint geodesic balls diffeomorphic to $B^{n}$.

Now, we take $\Omega$ constructed from Lemma \ref{counter;example;nontrivial,block}. After a scaling, we may assume that $\overline{\Omega} \subset B^{n}$, where $B^{n}$ is an open solid round ball of radius 1. Let $\Omega^c = B^n \backslash \Omega$.

We can find diffeomorphisms $\varphi_i: {B^{n}}\rightarrow {B_i}\subset M_1$ ($i= 1,2$) such that
\begin{align}\label{counter-example;diffeomorphism-derivative-bound}
    |D\varphi_i| + |D^2\varphi_i| < 100.
\end{align}
Let $\Omega_i = \varphi_i(\Omega)\subset B_i$ be a smooth open domain in $M_1$. 

Finally, we define 
\begin{align*}
	M = M_1 \backslash (\Omega_1\sqcup \Omega_2).
\end{align*}

We shall prove that $M$ satisfies all the claimed properties.
First, direct computation shows that $M_1$ satisfies $R - \frac{1}{8} g\owedge g \in C_{PIC}$. Thus $M$ also satisfies the same curvature condition. 

By construction, $\partial M = \partial\Omega_1 \cup \partial\Omega_2$ and $\partial \Omega_1, \partial \Omega_2$ are both connected. Moreover, 
\begin{align*}
	d(\partial \Omega_1, \partial \Omega_2) > d(B_1, B_2) = 2 L -2 > L.
\end{align*}  
This proves the width lower bound.

Since $\Omega$ is a fixed smooth domain, there is a constant $C'$ such that $|A| \leq C'$ on $\partial \Omega$. By (\ref{counter-example;diffeomorphism-derivative-bound}), the boundary curvature of $\partial \Omega_i$ satisfies $|A| \leq C$ for some universal constant $C = C(n,k)$.

It remains to check the topology. To this end, we rewrite $M$ as
\begin{align*}
	M = \big( M_1 \backslash (B_1 \sqcup B_2) \big) \cup (B_1 \backslash \Omega_1) \cup (B_2 \backslash \Omega_2) . 
\end{align*}
By construction, each of $B_i \backslash \Omega_i$ is connected, thus $M$ is connected.

Note that 
\begin{align*}
	\overline{B_1 \backslash \Omega_1} \cap \overline{B_2 \backslash \Omega_2} =& \phi 
\end{align*} 
and
\begin{align*}
	\overline{\big( M \backslash (B_1 \sqcup B_2) \big)} \cap \overline{B_i \backslash \Omega_i} \simeq & S^{n-1}  \quad (i=1,2).
\end{align*} 
The Mayer-Vietoris sequence then implies that
 \begin{align*}
    H^k(M) \cong &H^k(B_2 \backslash \Omega_1)\oplus H^k(B_2 \backslash \Omega_2)\oplus H^k(M_1 \backslash (B_1 \sqcup B_2) ) \\
    \cong &H^k(\Omega^c)\oplus H^k(\Omega^c)\oplus H^k(M_1 \backslash (B_1 \sqcup B_2) )
\end{align*}
By Lemma \ref{topology;punctuation}, $b_k(M_1 \backslash (B_1 \sqcup B_2)) = b_k(M_1 \backslash B_1) = b_k(M_1) = 0$. 
By Lemma \ref{counter;example;nontrivial,block}, $b_k(\Omega^c) = 1$. 

Thus, we conclude that:
\begin{align*}
	b_k(M) = 2 b_k(\Omega^c) + b_k(M_1 \backslash (B_1 \sqcup B_2)) = 2
\end{align*}
verifying all the claimed properties of $M$.

\end{proof}

 
\newpage

\appendix

\section{Calculations for the Weitzenb\"ock curvature term}

\begin{lemma}
	The curvature term $\mathcal{R}$ in the Weitzenb\"ock formula satisfies
	\[	\mathcal{R}(\omega) = -\theta^i\wedge i_{e_j}R(e_i, e_j)\omega,\]
	where $\{e_1,..,e_n\}$ is an orthonormal frame on $TM$.
\end{lemma}

\begin{proof}
	Using the definition of $c$, we compute
	
	\begin{align*}
		&\sum_{i,j=1}^n c(e_i)c(e_j)R(e_i, e_j)\omega\\
		&= \underbrace{\sum_{i,j=1}^n \theta^i\wedge\theta^j\wedge R(e_i, e_j)\omega}_{=0} + \underbrace{\sum_{i,j=1}^n i_{e_i}i_{e_j}R(e_i, e_j)\omega}_{=0}  -\sum_{i,j=1}^n\theta^i\wedge i_{e_j}R(e_i, e_j)\omega \\
		&\quad- \sum_{i,j=1}^n i_{e_i}(\theta^j\wedge R(e_i, e_j)\omega)\\
		&= -2\sum_{i,j=1}^n\theta^i\wedge i_{e_j}R(e_i, e_j)\omega.
	\end{align*}
\end{proof}

\begin{lemma}
	Suppose that $\omega\in \wedge^2 T^*M$ is a two-form, then 
	\[	(R(X,Y)\omega)(Z,W) = -\omega(R(X,Y)Z,\, W) - \omega(Z,\, R(X,Y)W).\]
	for vector fields $X,Y,Z,W\in TM$.
\end{lemma}

\begin{corollary}\label{curv:term:calc}
	Suppose that $\{e_1,\dots, e_n\}$ is an orthonormal frame on $TM$, and $\{\theta^1,\dots,\theta^n\}$ is the dual frame. Then  
	\[	\mathcal{R}(\theta^i\wedge\theta^j) = R_{klil}\,  \theta^k\wedge \theta^j - R_{kljl}\,   \theta^k\wedge \theta^i -  2R_{ikjl}\,   \theta^k\wedge \theta^l.\]
	In particular, the operator $\mathcal{R}:\wedge^2(T^*M)\to \wedge^2(T^*M) $ is a self-adjoint operator with respect to the inner product $\langle\, ,\, \rangle$.
\end{corollary}

\begin{proof}
	By the previous Lemma, we have
	\[	R(e_i, e_j) (\theta^k\wedge\theta^l) = R_{ijpk}\, \theta^p\wedge \theta^l + R_{ijpl}\, \theta^k\wedge \theta^p.\]
So
	
	\begin{align*}
		\mathcal{R}(\theta^k\wedge\theta^l) &= -\theta^i\wedge i_{e_j}\Big( R_{ijpk}\, \theta^p\wedge \theta^l + R_{ijpl}\,\theta^k\wedge \theta^p\Big)\\
		&= -R_{ijjk}\,  \theta^i\wedge \theta^l + R_{ilpk}\,   \theta^i\wedge \theta^p -  R_{ikpl}\,   \theta^i\wedge \theta^p + R_{ijjl}\,   \theta^i\wedge \theta^k\\
		&= R_{ijkj}\,  \theta^i\wedge \theta^l - R_{ijlj}\,   \theta^i\wedge \theta^k -  2R_{ikpl}\,   \theta^i\wedge \theta^p.
	\end{align*}

\end{proof}

\bigskip

\section{Hessian and Laplace comparison}\label{section;comparison}
Suppose that $\Sigma^{n-1} \subset M^n$ is a hypersurface. Consider the distance function $\rho = \rho(x,\Sigma)$. 
\begin{proposition}\label{Hessian upper bound}
Given $x\in M$ and choose a minimizing geodesic $\gamma: [0,1]\rightarrow M$ from $\Sigma$ to $x$ with $\gamma(0)\in \Sigma$ and $\gamma(1) = x$. Then 
\begin{align}\label{Hessian of dist2}
    \frac{1}{2}\nabla^2 \rho^2 (e_i,e_i)\leq \int_0^1 |D_t V|^2 -R(V, \gamma',V,\gamma')dt - A(V(0),V(0))\rho,
\end{align}
where $V(t) \in T_{\gamma(t)}M$ is a vector field along $\gamma$ with $V(1) = e_i$ and $V(0) \in T_{\gamma(0)}\Sigma$. Moreover, $A(X,Y) = \langle \nabla_X \nu, Y\rangle$ and $\nu$ points in the direction of $-\gamma'(0)$.
\end{proposition}
\begin{proof}
    Let $F: (-\varepsilon, \varepsilon)\times [0,1]  \rightarrow M$ be a smooth map such that 
\begin{align*}
    F(0,t) =& \gamma(t), \quad F(s,0)\in\Sigma, \quad
    F(s,t) = \exp_{x}(te_i).
\end{align*}
Let $\gamma_s(t) = F(s,t)$ and consider the energy 
\begin{align*}
    E(s) = \int_0^1 |D_t F(s,t)|^2 dt.
\end{align*}
By Cauchy-Schwarz inequality, $E(s)\geq L(\gamma_s(t))^2$ with equality hold if and only if $|D_t F| = |\gamma_s'|$ is a constant in $t$. In particular, $E(0) = \rho(x)^2$ and consequently $E(s)$ serves as barrier for $\rho(F(s,1))$. 
\begin{align*}
    \nabla^2\rho^2(e_i,e_i) = \frac{d^2}{ds^2}\Big|_{s=0} \rho(F(s,1))^2 \leq E''(0).
\end{align*}
The identity on the left follows from the fact that $F(s,1)$ is a geodesic. 

Next, we compute the second derivative of $E$. To simplify notation, we let $V(t) = D_s F(0,t)$, so $V(1) = e_i$.  We shall also keep in mind that $D_tD_t F(0,t) = 0$.
\begin{align*}
    \frac{1}{2}E''(0) =& \int_0^1 |D_s D_t F|^2 + \langle D_sD_s D_t F, D_t F\rangle dt\\
    =& \int_0^1 |D_t D_s F|^2 + \langle D_tD_s D_s F, D_t F\rangle - R(V,\gamma', V, \gamma')dt \\
    =& \langle D_sD_s F,D_t F\rangle \Big|_{t=0}^{t=1} + \int_0^1 |D_t V|^2 - R(V,\gamma', V, \gamma')dt,
\end{align*}
where we used the geodesic equation $D_t D_t F(0, t) = 0$ in the last equality.

Since $F(0,1)$ is a geodesic, we have $D_sD_s F(s,1) =0$. So the boundary term at $t = 1$ vanishes.

By the first variation, $D_t F (0,0) \perp \Sigma$. Therefore $D_t F(0,0) = -\rho \nu$, where $\nu$ is a unit normal vector pointing away from $x$.  Then 
\begin{align*}
    \langle D_s D_s F, D_t F \rangle\Big|_{t= 0} = \langle D_sD_s F, -\rho \nu \rangle\Big|_{t= 0} = \langle D_s F, \rho D_s \nu\rangle\Big|_{t= 0} = \rho A(V(0),V(0)).
\end{align*}

\begin{align}
    \frac{1}{2}E''(0) \leq & -\rho A(V(0),V(0)) +  \int_0^1 |D_t V|^2 - R(V,\gamma', V, \gamma')dt .
\end{align}
\end{proof}

\begin{theorem}[Laplace comparison]\label{laplace;comparison}
If $Ric \geq -(n-1)K$ and $H_{\Sigma} \geq -\Lambda$, then 
\begin{align}\label{Laplace of dist}
    \Delta \rho \leq  (n-1)\sqrt{K} \cdot \frac{\Lambda + (n-1)\sqrt{K}\tanh(\sqrt{K}\rho)}{(n-1)\sqrt{K} + \Lambda\tanh(\sqrt{K}\rho)}.
\end{align} 
holds in the barrier sense.

Similarly, if $Sec \geq -K$ and $A_{\Sigma} \geq -\Lambda$, then 
\begin{align}
    \nabla^2 \rho \leq  \sqrt{K} \cdot \frac{\Lambda + \sqrt{K}\tanh(\sqrt{K}\rho)}{\sqrt{K} + \Lambda\tanh(\sqrt{K}\rho)}.
\end{align}
holds in the barrier sense.
\end{theorem}
\begin{proof}
    Choose a minimizing geodesic $\gamma$ such that $\gamma(0)\in \Sigma$ and $\gamma(1)  = x$. 
    Let us choose an orthonormal basis $e_1,...,e_{n-1}, \frac{\gamma'(1)}{\rho}$ of $T_x M$.
    We then extend $e_i$ parallelly along $\gamma$.
    By proposition \ref{Hessian upper bound},  we have
\begin{align}
    \sum_{i=1}^{n-1}\frac{1}{2}\nabla^2 \rho^2 (e_i,e_i)\leq \sum_{i=1}^{n-1}\frac{1}{2}\int_0^1 |D_t V_i|^2 -R(V_i, \gamma',V_i,\gamma')dt - A(V_i(0),V_i(0))\rho.
\end{align}
where $V_i(1) = e_i$ and $V_i(0) \in T_{\gamma(0)}(\Sigma)$. 

It is also clear that

\begin{align*}
    \frac{1}{2}\nabla^2 \rho^2 \left(\frac{\gamma'}{|\gamma'|},\frac{\gamma'}{|\gamma'|}\right)\leq 1.
\end{align*}

Now we choose $V_i = f(t)e_i(t)$ with $f(1) = 1$. 
Plugging into the formula we have
\begin{align}\label{Hessian}
    \sum_{i=1}^{n-1}\frac{1}{2}\nabla^2 \rho^2 (e_i,e_i)\leq& \int_0^1 (n-1)f'^2 -Ric( \gamma',\gamma')f^2dt - f^2(0)H_{\Sigma}(\gamma(0))\rho \\
    \leq& \int_0^1 \Big((n-1)f'^2 + (n-1)K\rho^2 f^2\Big)dt + f^2(0)\Lambda\rho .\nonumber
\end{align}
In order to minimize the above expression, we choose $f$ such that 
\begin{align}
    (n-1)f'' - (n-1)K \rho^2 f =& 0,\\
    (n-1)f'(0) - f(0)\Lambda\rho =& 0  .
\end{align}
Hence we shall take
\begin{align}
    f(t) = \frac{\sinh(\lambda t) + \mu \cosh(\lambda t)}{\sinh(\lambda) + \mu \cosh(\lambda )},
\end{align}
where 
\begin{align*}
    \lambda= \sqrt{K}\rho, \quad \mu = \frac{(n-1)\lambda}{\Lambda\rho} = (n-1)\frac{\sqrt{K}}{\Lambda}.
\end{align*}
are both scaling invariant.
 
Plugging into the hessian formula we get:
\begin{align*}
    &\sum_{i=1}^{n-1}\frac{1}{2} \nabla^2 \rho^2 (e_i,e_i)\\
    &\leq \big(\sinh(\lambda) + \mu \cosh(\lambda )\big)^{-2} 
          \cdot \Big[(n-1)\lambda\Big(\frac{(\mu^2+1)}{2}\sinh(2\lambda ) + \mu \cosh(2\lambda )\Big) - (n-1)\lambda \mu + \mu^2 \Lambda\rho\Big].
\end{align*}
Note that
\begin{align*}
    -(n-1)\lambda \mu + \mu^2 \Lambda\rho = 0
\end{align*}
and 
\begin{align*}
    &\frac{(\mu^2+1)}{2}\sinh(2\lambda ) + \mu \cosh(2\lambda ) = \cosh^2(\lambda) \left[(\mu^2+1)\tanh(\lambda) + \mu + \mu\tanh^2(\lambda)\right].
\end{align*}
we obtain that:
\begin{align*}
    \sum_{i=1}^{n-1}\frac{1}{2} \nabla^2 \rho^2 (e_i,e_i) \leq& (n-1)\lambda\cdot\frac{(\mu^2+1)\tanh(\lambda) + \mu + \mu\tanh^2(\lambda)}{(\mu+ \tanh(\lambda))^2}\\
    =& (n-1)\sqrt{K}\rho \frac{\Lambda + (n-1)\sqrt{K}\tanh(\sqrt{K}\rho)}{(n-1)\sqrt{K} + \Lambda\tanh(\sqrt{K}\rho)}.
\end{align*}
This implies that
\begin{align*}
    \Delta \rho =& \nabla^2 \rho \left(\frac{\gamma'(1)}{|\gamma'(1)|}, \frac{\gamma'(1)}{|\gamma'(1)|}\right)  + \sum_{i=1}^{n-1}\frac{1}{2} \nabla^2 \rho (e_i,e_i) \\
    \leq& (n-1)\sqrt{K}\cdot \frac{\Lambda + (n-1)\sqrt{K}\tanh(\sqrt{K}\rho)}{(n-1)\sqrt{K} + \Lambda\tanh(\sqrt{K}\rho)}.
\end{align*}

To prove the Hessian comparison assuming $A\geq -\Lambda$ and $Sec\geq -K$, we have
\begin{align} 
    \frac{1}{2}\nabla^2 \rho^2 (e_i,e_i)
    \leq& \int_0^1 \Big(f'^2 + K\rho^2 f^2\Big)dt + f^2(0)\Lambda\rho .\nonumber
\end{align}
We then solve the ODE:
\begin{align}
     f'' - K \rho^2 f =& 0,\\
    f'(0) - f(0)\Lambda\rho =& 0  .
\end{align}
Hence we shall take
\begin{align}
    f(t) = \frac{\sinh(\lambda t) + \mu \cosh(\lambda t)}{\sinh(\lambda) + \mu \cosh(\lambda )},
\end{align}
where 
\begin{align*}
    \lambda= \sqrt{K}\rho, \quad \mu = \frac{\lambda}{\Lambda\rho} = \frac{\sqrt{K}}{\Lambda}.
\end{align*}
Now we use the same computation by replacing $\Lambda$ by $(n-1)\Lambda$, and the result follows.
\end{proof}

\bigskip

\begin{theorem}[Hessian/Laplace comparison upper bound, positive boundary]\label{hess;rho;upper;bound;positive;boundary}
If $sec \geq -K$. Suppose that at some point $p\in \Sigma$ and unit vector $e \in T_p\Sigma$ we have $A(e_i,e_i) \geq \Lambda$.
Suppose that geodesic $\gamma(t) = \exp_p(-\rho t\nu)$ is minimizing for $t\in [0,1]$, Then at $ \exp_p(-\rho\nu) = \gamma(1) $:
\begin{align}\label{hessian;upper;convex boundary}
    \inf\limits_{|v| = 1}\nabla^2 \rho(v,v) \leq  \sqrt{K} \frac{\sqrt{K}\tanh(\sqrt{K}\rho)-\Lambda }{\sqrt{K} - \Lambda\tanh(\sqrt{K}\rho)}
    \end{align} 
holds in the barrier sense.

If $Ric \geq -(n-1)K$ and $H \geq (n-1)\Lambda$, then at $ \exp_p(-\rho\nu) = \gamma(1) $: 
\begin{align}\label{hessian;upper;convex boundary}
    \Delta \rho \leq  (n-1)\sqrt{K} \frac{\sqrt{K}\tanh(\sqrt{K}\rho)-\Lambda }{\sqrt{K} - \Lambda\tanh(\sqrt{K}\rho)}
    \end{align} 
\end{theorem}
holds in the barrier sense.

\begin{proof}
We only prove the case $sec \geq -K$. The proof in the case $Ric \geq -(n-1)K$ is similar.

Assume that $\gamma(1)  = x$. 
    Let us  extend $e_i$ parallelly along $\gamma$ and call it $e_i(t)$.
    By proposition \ref{Hessian upper bound},  we have
\begin{align}
    \frac{1}{2}\nabla^2 \rho^2 (e_i,e_i)\leq \frac{1}{2}\int_0^1 |D_t V_i|^2 -R(V_i, \gamma',V_i,\gamma')dt - A(V_i(0),V_i(0))\rho.
\end{align}
where $V_i(1) = e_i$ and $V_i(0) \in T_{\gamma(0)}(\Sigma)$. 

Now we choose $V_i = f(t)e_i(t)$ with $f(1)=1$. Plugging into the above formula we get:
\begin{align*}
	\frac{1}{2}\nabla^2 \rho^2 (e_i,e_i)\leq \int_0^1 \big(f'^2 + K\rho^2 f^2 \big)dt - f^2(0)\Lambda\rho.
\end{align*}
In order to minimize the above expression, we choose $f$ such that 
\begin{align}
    f'' - K \rho^2 f =& 0,\\
    f'(0) + f(0)\Lambda\rho =& 0  .
\end{align}

Hence we shall take
\begin{align}
    f(t) = \frac{\sinh(\lambda t) - \mu \cosh(\lambda t)}{\sinh(\lambda) - \mu \cosh(\lambda )},
\end{align}
where 
\begin{align*}
    \lambda= \sqrt{K}\rho, \quad \mu = \frac{\lambda}{\Lambda\rho} = \frac{\sqrt{K}}{\Lambda}.
\end{align*}

Plugging into the hessian formula we get:
\begin{align*}
    &\sum_{i=1}^{n-1}\frac{1}{2} \nabla^2 \rho^2 (e_i,e_i)\\
     &\leq  \big(\sinh(\lambda) - \mu \cosh(\lambda )\big)^{-2}\Big[\lambda\Big(\frac{(\mu^2+1)}{2}\sinh(2\lambda ) - \mu \cosh(2\lambda )\Big) + \lambda \mu - \mu^2 \Lambda\rho\Big].
\end{align*}
Note that
\begin{align*}
	\lambda \mu - \mu^2\Lambda \rho = 0
\end{align*}
and 
\begin{align*}
    &\frac{(\mu^2+1)}{2}\sinh(2\lambda ) - \mu \cosh(2\lambda ) \\
    =& (\mu^2 + 1)\sinh(\lambda)\cosh(\lambda) - \mu \cosh^2(\lambda) - \mu\sinh^2(\lambda)  \\
    =&\cosh^2(\lambda) \left[(\mu^2+1)\tanh(\lambda) - \mu - \mu\tanh^2(\lambda)\right].
\end{align*}

Moreover
\begin{align*}
    (\sinh(\lambda)-\mu\cosh(\lambda))^2 = \cosh^2(\lambda) (\mu - \tanh(\lambda))^2.
\end{align*}

Putting everthing together we obtain that 
\begin{align*}
    \frac{1}{2} \nabla^2 \rho^2 (e_i,e_i) \leq& \lambda\frac{(\mu^2+1)\tanh(\lambda) - \mu - \mu\tanh^2(\lambda)}{(\mu - \tanh(\lambda))^2}\\
    =& \sqrt{K}\rho \frac{\sqrt{K}\tanh(\sqrt{K}\rho)-\Lambda }{\sqrt{K} - \Lambda\tanh(\sqrt{K}\rho)}.
\end{align*}
\end{proof}

\bigskip

\begin{theorem}[Hessian comparison within focal radius]\label{hessian;lower;focal}
If $sec \geq -K$ and
$d(x,\Sigma) \leq \frac{r_f}{2}$ then 
\begin{align}
    \nabla^2 \rho \geq  -\frac{\sqrt{K}}{\tanh(\frac{r_f \sqrt{K}}{2})}.
\end{align} 
\end{theorem}

\begin{proof}
Suppose that the estimate fails at some $x_0$ with $\rho(x_0) = \rho_0 < \frac{r_f}{2}$, pick a minimizing geodesic ${\gamma}$ connecting $\Sigma$ and $x_0$ with $p = {\gamma}(0) \in \Sigma$ and ${\gamma}(1) = x_0$. 
Note that $|{\gamma}'| = \rho_0$.

Now we consider extension: let 
\begin{align*}
	\tilde{\gamma}(t) = \exp_p((t+1){\gamma}'(0)).
\end{align*}
By definition of focal radius, $\gamma$ is a minimizing geodesic connecting $\gamma(t)$ and $\Sigma$ when 
$(t+1)\rho_0<r_f$.
In particular if we define $\tilde{\rho} = d(x, \Sigma_{\rho_0})$, then
\begin{align*}
	\tilde{\rho} = \rho - \rho_0
\end{align*}
when $\rho < r_f$. 

Let us define the level set 
\begin{align*}
	\Sigma_{r} =  \{x\in M: \rho(x) = r\}.
\end{align*}
By the definition of focal radius, when $\rho < r_f $, $\Sigma_r$ is a smooth hypersurface diffeomorphic to $\Sigma$  
and  $\rho $ is a smooth function. In particular,
\begin{align*}
	|\nabla \rho| = 1.
\end{align*}

The second fundamental form on $\Sigma$ with respect to $\nu = -\frac{\nabla \rho}{|\nabla \rho|}$ is
\begin{align*}
	A(X,Y) = \left\langle -\nabla_X \left(\frac{\nabla \rho}{|\nabla \rho|}\right),Y\right\rangle = -\nabla^2\rho(X,Y),
\end{align*}
where $X,Y$ are tangent to $\Sigma_r$. 

By assumption, there exists unit vector $e\in T_{x_0} \Sigma_r$ such that 
\begin{align*}
	A(e,e) = -\nabla^2\rho(e,e) \geq \frac{\sqrt{K}}{\tanh(\frac{r_f \sqrt{K}}{2})} =: \Lambda_1.
\end{align*}

However, by Hessian comparison Proposition (\ref{hess;rho;upper;bound;positive;boundary}) we have
\begin{align*}
	\inf_{|v|=1}\nabla^2\rho(v,v) = \inf_{|v|=1}\nabla^2\tilde{\rho}(v,v)  \leq \sqrt{K} \frac{\sqrt{K}\tanh(\sqrt{K}\tilde{\rho})-\Lambda_1}{\sqrt{K} - \Lambda_1\tanh(\sqrt{K}\tilde{\rho})},
\end{align*}
where $\tilde{\rho}$ is allowed to take any value in $[0,\frac{r_f}{2})$. Now we let 
$\tilde{\rho} \rightarrow \frac{r_f}{2}^{-}$, then  
\begin{align*}
	\sqrt{K}\tanh(\sqrt{K}\tilde{\rho}) - \Lambda_1 \rightarrow \sqrt{K}(\tanh(\sqrt{K}r_f/2) - \tanh(\sqrt{K}r_f/2)^{-1}) < 0
\end{align*}
and 
\begin{align*}
	\sqrt{K} -\Lambda_1 \tanh(\sqrt{K}\tilde{\rho})\rightarrow 0^{+}.
\end{align*}
Therefore, 
\begin{align*}
	\inf_{|v|=1}\nabla^2\tilde{\rho}(v,v) \leq \sqrt{K} \frac{\sqrt{K}\tanh(\sqrt{K}\tilde{\rho})-\Lambda_1}{\sqrt{K} - \Lambda_1\tanh(\sqrt{K}\tilde{\rho})}\rightarrow -\infty.
\end{align*}
However, $\rho = \tilde{\rho} + \rho_0 \leq \frac{r_f}{2} + \rho_0 < r_f$. By definition of focal radius, $\nabla^2\rho$ is bounded, this is a contradiction.  
\end{proof}

\bigskip

\begin{theorem}[Laplace comparison within focal radius]\label{laplace;lower;focal}
If $Ric \geq -(n-1)K$ and
$d(x,\Sigma) \leq \frac{r_f}{2}$ then 
\begin{align}
    \Delta \rho \geq  -\frac{(n-1)\sqrt{K}}{\tanh(\frac{r_f \sqrt{K}}{2})}.
\end{align} 
\end{theorem}

\begin{proof}
Suppose that the estimate fails at some $x_0$ with $\rho(x_0) = \rho_0 < \frac{r_f}{2}$, 
let us set up as in Theorem \ref{hessian;lower;focal}.  

The mean curvature on $\Sigma$ with respect to $\nu = -\frac{\nabla \rho}{|\nabla \rho|}$ is
\begin{align*}
	H = \sum_{i=1}^{n-1}\left\langle -\nabla_{e_i} \left(\frac{\nabla \rho}{|\nabla \rho|}\right),e_i\right\rangle = -\Delta\rho,
\end{align*}
where $e_i$ are orthonormal frame on $\Sigma_r$. By our assumption,
\begin{align*}
	H = -\Delta \rho \geq \frac{(n-1)\sqrt{K}}{\tanh(\frac{r_f \sqrt{K}}{2})} =: (n-1)\Lambda_1.
\end{align*}

However, by Laplace (\ref{hess;rho;upper;bound;positive;boundary}) we have
\begin{align*}
	\Delta\rho(v,v) = \Delta\tilde{\rho}(v,v)  \leq (n-1)\sqrt{K} \frac{\sqrt{K}\tanh(\sqrt{K}\tilde{\rho})-\Lambda_1}{\sqrt{K} - \Lambda_1\tanh(\sqrt{K}\tilde{\rho})},
\end{align*}
where $\tilde{\rho}$ is allowed to take any value in $[0,\frac{r_f}{2})$. Now we let 
$\tilde{\rho} \rightarrow \frac{r_f}{2}^{-}$, then  
\begin{align*}
	\Delta \tilde{\rho}(v,v) \leq (n-1)\sqrt{K} \frac{\sqrt{K}\tanh(\sqrt{K}\tilde{\rho})-\Lambda_1}{\sqrt{K} - \Lambda_1\tanh(\sqrt{K}\tilde{\rho})}\rightarrow -\infty.
\end{align*}
However, $\rho = \tilde{\rho} + \rho_0 \leq \frac{r_f}{2} + \rho_0 < r_f$. By definition of focal radius, $\Delta\rho$ is bounded, this is a contradiction.  
\end{proof}

\bigskip

\begin{corollary}\label{hessian;laplce;lower;focal;positive;curvature}
If $sec \geq 0$ and $d(x,\Sigma) \leq \frac{r_f}{2}$, then
	\begin{align*}
		\nabla^2\rho \geq -\frac{2}{r_f}.
	\end{align*}
If $Ric \geq 0$ and $d(x,\Sigma) \leq \frac{r_f}{2}$, then
	\begin{align*}
		\Delta\rho \geq -\frac{2(n-1)}{r_f}.
	\end{align*}
\end{corollary}

\end{document}